\documentclass[11pt]{article} % Anonymized

\usepackage{multirow}% http://ctan.org/pkg/multirow
\usepackage{color}
\usepackage{comment}
\usepackage{amsfonts}
\usepackage{array}

\usepackage{fullpage}

\usepackage{amsmath, amsfonts, amsthm, amssymb}

\usepackage{algorithm}
\usepackage[noend]{algpseudocode}

\usepackage{hyperref}  
\hypersetup{colorlinks=true,citecolor=blue,linkcolor=blue} 

\newcommand{\poly}{\mathrm{poly}}

\newcommand{\wt}[1]{\tilde{#1}}

\DeclareMathOperator{\var}{Var}
\DeclareMathOperator{\nnz}{nnz}
\newcommand{\inorm}[1]{\|{#1}\|_{\infty}}

\newcommand{\norm}[1]{\|{#1}\|}

\newcommand{\cS}{\mathcal{S}}
\newcommand{\cM}{\mathcal{M}}
\newcommand{\bP}{\boldsymbol{P}}
\newcommand{\bSigma}{\boldsymbol{\Sigma}}

\newcommand{\br}{\boldsymbol{r}}
\newcommand{\cA}{\mathcal{A}}

\newcommand{\bI}{\boldsymbol{I}}

\newcommand{\cT}{\mathcal{T}}
\newcommand{\bw}{\boldsymbol{w}}
\newcommand{\bx}{\boldsymbol{x}}
\newcommand{\bp}{\boldsymbol{p}}
\newcommand{\bsigma}{\boldsymbol{\sigma}}

\newcommand{\bxi}{\boldsymbol{\xi}}

\newcommand{\bg}{\boldsymbol{g}}
\newcommand{\bQ}{\boldsymbol{Q}}

\newcommand{\EE}{\mathbb{E}}
\newcommand{\cN}{\mathcal{N}}
\newcommand{\RR}{\mathbb{R}}

\newcommand{\bv}{\boldsymbol{v}}
\newcommand{\bu}{\boldsymbol{u}}
\newcommand{\one}{\boldsymbol{1}}

\newcommand{\wh}[1]{\widehat{#1}}
\newcommand{\tO}{\tilde{O}}
\newcommand{\cK}{\mathcal{K}}

\newtheorem{fact}{Fact}

\newtheorem{theorem}{Theorem}[section]
\newtheorem{lemma}[theorem]{Lemma}
\newtheorem{corollary}[theorem]{Corollary}

\theoremstyle{definition}
\newtheorem{definition}[theorem]{Definition}

\newtheorem{proposition}{Proposition}[theorem]

\theoremstyle{remark}
\newtheorem{remark}[theorem]{Remark}

\makeatletter
\def\BState{\State\hskip-\ALG@thistlm}
\makeatother

\numberwithin{equation}{section}

\newcommand{\R}{\mathbb{R}}

\newcommand{\states}{\cS}
\newcommand{\actions}{\cA}
\newcommand{\valuespace}{\R^{\states}}
\newcommand{\policyspace}{\actions^\states}

\newcommand{\defeq}{\stackrel{\mathrm{\scriptscriptstyle def}}{=}}

\newcommand{\otilde}{\tilde{O}}
\newcommand{\vopt}{\bv^{*}}

\DeclareMathOperator*{\argmax}{argmax}

\def\cT{\mathcal{T}}
\def\bP{{\boldsymbol{P}}}
\def\br{{\boldsymbol{r}}}
\def\bu{{\boldsymbol{u}}}
\def\bv{{\boldsymbol{v}}}
\newcommand{\halfErr}{\textsc{HalfErr}}

\def\RR{\mathbb{R}}

\title{Near-Optimal Time and Sample Complexities for Solving Discounted Markov Decision Process with a Generative Model}
\author{
	Aaron Sidford \\
	Stanford University \\ 
	\texttt{sidford@stanford.edu}
	\and
	Mengdi Wang \\
	Princeton University \\
	\texttt{mengdiw@princeton.edu}
	\and
	Xian Wu \\
	Stanford University \\
	\texttt{xwu20@stanford.edu}
	\and
	Lin F. Yang \\
	Princeton University \\
	\texttt{lin.yang@princeton.edu}
	\and
	Yinyu Ye \\
	Stanford University \\
	\texttt{yyye@stanford.edu}
}

\begin{document}

\maketitle

\begin{abstract}
In this paper we consider the problem of computing an $\epsilon$-optimal policy of a discounted Markov Decision Process (DMDP) provided we can only access its transition function through a generative sampling model that given any state-action pair samples from the transition function in $O(1)$ time. Given such a DMDP with states $\states$, actions $\actions$, discount factor $\gamma\in(0,1)$, and rewards in range $[0, 1]$ we provide an algorithm which computes an $\epsilon$-optimal policy with probability $1 - \delta$ where {\it both} the time spent and number of sample taken are upper bounded by 
	\[
O\left[\frac{|\cS||\cA|}{(1-\gamma)^3 \epsilon^2} \log \left(\frac{|\cS||\cA|}{(1-\gamma)\delta \epsilon}
		\right) 
		\log\left(\frac{1}{(1-\gamma)\epsilon}\right)\right] ~.
\]
For fixed values of $\epsilon\in(0,1)$, this improves upon the previous best known bounds by a factor of $(1 - \gamma)^{-1}$ and matches the sample complexity lower bounds proved in \cite{azar2013minimax} up to logarithmic factors. 
We also extend our method to computing $\epsilon$-optimal policies for finite-horizon MDP with a generative model and provide a nearly matching sample complexity lower bound. 
\end{abstract} 
% I commented out the use of small, are we sure can do that without breaking formatting rules?
% as well as this paper's results on computing $\epsilon$-approximate value functions. 
%We also provide matching upper and lower sample bounds for obtaining policy of $H$-finite horizon MDP with a generative model.
%\mw{Shall we give [AKM13] more credit in the abstract? Or shall we not mention [AKM13] explicitly in the abstract to avoid hard feelings?}

 %\mw{Note finite-horizon result is sample optimal but not time optimal}

%We achieve these result by combining variance reduction techniques from \cite{sidford2018variance} with variance analysis techniques from \cite{azar2013minimax} in a way that we hope opens the door for further research in reinforcement learning. 

\newpage

\section{Introduction}
%\sidford{
%
%Markov Decision Processes ... fundamental
%
%In the case when can just sample ... natural also fundamental
%
%Long line of work to establish sample complexity.
%
%What's known? Optimal sample complexity (up to logs) for estimating value and then worse bounds for everything else.
%
%Make sure to emphasize that our algorithms are not sampling dynamically. We just need enough samples from each state.
%}

%Markov decision process (MDP) 
%is a fundamental problem regarding sequentially making decisions in a dynamic environment. It serves as the basic model of discrete-time stochastic control and reinforcement learning. Computing the optimal policy of MDP finds major applications in computer science and operations research. From a theoretical point of view, it is of vital interest to understand the computational complexity of Markov decision process. From a  practical point of view, efficient algorithms with sharp guarantees are in demand.

%Learning the optimal policy of an unknown MDP is a central problem in reinforcement learning (RL). The challenge of RL is to teach an agent to behave optimally in an unknown environment. In typical RL settings, knowledge of immediate and delayed rewards, as well as probabilities of state-to-state transitions using different actions are not known in advance but gradually learned through experience --  repeated interactions with the environment itself. A first and foremost theoretical question is how many sample transitions one needs to observe from the environment before making approximately optimal decisions. 
Markov decision processes (MDPs) are a fundamental mathematical abstraction used to model sequential decision making under uncertainty and are a basic model of discrete-time stochastic control and reinforcement learning (RL). %Consequently, obtaining faster algorithms for optimizing over MDPs, i.e. computing optimal policies, is fundamental to the theory and practice of machine learning and computer science. 
Particularly central to RL is the case of computing or learning an approximately optimal policy when the MDP itself is not fully known beforehand. %In RL we wish create agents which can behave optimally in unknown stochastic environments, the structure of which is only learned over time. 
One of the simplest such settings is when the states, rewards, and actions are all known but the transition between states when an action is taken is probabilistic, unknown, and can only be sampled from. 

Computing an approximately optimal policy with high probability in this case is known as PAC RL with a generative model. It is a well studied problem with multiple existing results providing algorithms with improved the sample complexity (number of sample transitions taken) and running time (the total time of the algorithm) under various MDP reward structures, e.g. discounted infinite-horizon, finite-horizon, etc. (See Section~\ref{sec:prev_work} for a detailed review of the literature.)

In this work, we consider this well studied problem of computing approximately optimal policies of discounted infinite-horizon Markov Decision Processes (DMDP) under the assumption we can only access the DMDP by sampling state transitions. Formally, we suppose that we have a DMDP with a known set of states, $\states$, a known set of actions that can be taken at each states, $\actions$, a known reward $\br_{s,a} \in [0, 1]$ for taking action $a \in \actions$ at state $s \in \states$, and a discount factor $\gamma\in(0,1)$. We assume that taking action $a$ at state $s$ probabilistically transitions an agent to a new state based on a fixed, but unknown probability vector $\bP_{s, a}$. The objective is to maximize the cumulative sum of discounted  rewards in expectation.
Throughout this paper, we assume that we have a \emph{generative model}, a notion introduced by \cite{kakade2003sample}, which allows us to draw random state transitions of the DMDP. In particular, we assume that we can sample from the distribution defined by $\bP_{s, a}$ for all $(s,a) \in \cS \times \cA$ in $O(1)$ time. This is a natural assumption and can be achieved in expectation in certain computational models with linear time preprocessing of the DMDP.\footnote{If instead the oracle needed time $\tau$, every running time result in this paper should be multiplied by $\tau$.}
%In the reinforcement learning literature, e.g. \cite{kakade2003sample}, the sampling oracle corresponds to a ``generative model'' (a simulator of the MDP).
%We assume we can only access this vector by sampling from a generative model, that is we are given an oracle which can compute an independent sample according to the probability vector in $O(1)$ time. Our goal is to efficiently compute an approximately optimal policy, i.e. action to take at each state, that minimizes the expected cumulative discounted rewards over time when following this policy. (See Section~\ref{sec:prelim} for a more precise explanation)  

The main result of this paper is that we provide the first algorithm that is sample-optimal and runtime-optimal (up to polylogarithmic factors) for computing an $\epsilon$-optimal policy of a DMDP with a generative model (in the regime of $ 1/\sqrt{(1-\gamma)|\cS|}\le \epsilon\le 1$). In particular, we develop a randomized Variance-Reduced Q-Value Iteration (vQVI) based algorithm that computes an $\epsilon$-optimal policy with probability $1 - \delta$ with a number of samples, i.e. queries to the generative model, bound by
\[
		O\left[\frac{|\cS||\cA|}{(1-\gamma)^3 \epsilon^2} \log \left(\frac{|\cS||\cA|}{(1-\gamma)\delta \epsilon}
		\right) 
		\log\left(\frac{1}{(1-\gamma)\epsilon}\right)\right] ~.
\]
This result matches (up to polylogarithmic factors) the following sample complexity lower bound established in \cite{azar2013minimax} for finding $\epsilon$-optimal policies with probability $1-\delta$ (see Appendix~\ref{sec:lower bound}):
$$\Omega \left[\frac{|\cS||\cA|}{(1-\gamma)^3 \epsilon^2} \log\left(\frac{|\cS||\cA| }{\delta}\right) \right] ~.$$
Furthermore, we show that the algorithm can be implemented using sparse updates such that the overall run-time complexity is equal to its sample complexity up to constant factors, as long as each sample transition can be generated in $O(1)$ time. Consequently, up to logarithmic factors our run time complexity is optimal as well. In addition, the algorithm's %, we show it is sufficient for the algorithm to keep track of a constant number of values for each state-action pair without having to store all the samples. 
 space complexity is $\Theta(|\cS||\cA|)$.

%Due to the aforemention sample complexity lower bound of $\Omega(\frac{|\cS||\cA|}{(1-\gamma)^3 \epsilon^2} \log(\frac{|\cS||\cA| }{\delta}))$ from \cite{azar2013minimax} (See Appendix~\ref{sec:lower bound}), our algorithm is the first to achieve nearly optimal sample complexity with a generative model.  See Section~\ref{sec:prev_work} for an in-depth comparison with previous results.
%\lin{check this:}
%Note that \cite{azar2013minimax} obtains similar a bound for policy but only for $\epsilon=O(1/\sqrt{S})$.
%\mw{Azar 13 also used the QVI method (the exact version), which might make our method seem similar to theirs. But our method is actually far more sophisticated than QVI. Shall we think of another name to distinguish?}

%We provide an algorithm whose sample complexity matches (up to polylogarithmic factors) the lower bound established in \cite{azar2013minimax} as well as the upper bound proved in that paper for approximately learning the value function.

Our method and analysis builds upon a number of prior works. (See Section~\ref{sec:prev_work} for an in-depth comparison.)
The paper \cite{azar2013minimax} provided the first algorithm that achieves the optimal sample complexity for finding $\epsilon$-optimal value functions (rather than $\epsilon$-optimal policy), as well as the matching lower bound. Unfortunately an $\epsilon$-optimal value function does not imply an $\epsilon$-optimal policy and if we directly use the method of \cite{azar2013minimax}  to get an $\epsilon$-optimal policy for constant $\epsilon$, the best known sample complexity is $\otilde(|\cS||\cA| (1-\gamma)^{-5} \epsilon^{-2})$.
\footnote{\cite{azar2013minimax} showed that one can obtain $\epsilon$-optimal \emph{value} $v$ (instead of $\epsilon$-optimal policy) using sample size $\propto (1-\gamma)^{-3}\epsilon^{-2}$. By using this $\epsilon$-optimal value $v$, one can get a greedy policy that is $[(1-\gamma)^{-1}\epsilon]$-optimal. By setting  $\epsilon\rightarrow (1-\gamma)\epsilon$, one can obtain an $\epsilon$-optimal policy, using the number of samples $\propto (1-\gamma)^{-5}\epsilon^{-2}$.}
This bound is known to be improvable through related work of \cite{sidford2018variance} which provides a method for computing an $\epsilon$-optimal policy using $\otilde(|\cS||\cA| (1-\gamma)^{-4} \epsilon^{-2})$ samples and total runtime and the work of \cite{azar2013minimax} which in the regime of small approximation error, i.e. where $\epsilon = O( (1-\gamma)^{-1/2}|\cS|^{-1/2})$,
%\sidford{Shouldn't there be a gamma?},
%\sidford{is a tilde needed?} 
 already provides a method that achieves the optimal sample complexity. However, when the approximation error takes fixed values, e.g. $\epsilon\geq\Omega((1-\gamma)^{-1/2}|\cS|^{-1/2})$,
 %\sidford{shouldn't there e a gamma here too?}
 there remains a gap between the best known runtime and sample complexity for computing an $\epsilon$-optimal policy and the theoretical lower bounds. For fixed values of $\epsilon$, which mostly occur in real applications, our algorithm improves upon the previous best sample and time complexity bounds by a factor of $(1-\gamma)^{-1}$ where $\gamma\in(0,1)$, the discount factor, is typically close to 1. 

We achieve our results by combining and strengthening techniques from both \cite{azar2013minimax} and \cite{sidford2018variance}. On the one hand, in \cite{azar2013minimax} the authors showed that simply constructing a ``sparsified" MDP model by taking samples and then solving this model to high precision yields a sample optimal algorithm in our setting for computing the approximate value of every state. On the other hand,  \cite{sidford2018variance} provided faster algorithms for solving explicit DMDPs and improved sample and time complexities given a sampling oracle. In fact, as we show in Appendix~\ref{sec:sparsify mdp}, simply combining these two results yields the first nearly optimal runtime for approximately learning the value function with a generative model. 
%We propose two methods in this paper: one for approximating the value and the other for approximating the policy. 
%The first method is based on a straightforward idea - construct a ``sparsified" MDP model by taking samples and then solve it to high precision. In particular, the sparsification procedure and its theoretical guarantees are shown in \cite{azar2013minimax}. To solve the sparsified MDP, we use the high-precision approximate value iteration algorithm introduced in \cite{sidford2018variance}, which achieves a small run time when the input is sparse. As a result, this method finds an approximate value function using a sample size and run-time equaling to $\tO(|\cS| |\cA|  (1-\gamma)^{-3}\epsilon^{-2})$, which matches the lower bound given in \cite{azar2013minimax}.\lin{The lower bound is for sample complexity, do we need to say something? I guess we can say lower bound on sample-based algorithms.} However, an approximate-optimal value function does not imply an approximate-optimal policy (see e.g. \cite{bertsekas2013abstract}). Therefore, unfortunately, our first method based solely on sparsification does not compute an $\epsilon$-optimal policy using optimal sample complexity. 
Unfortunately, it is known that an approximate-optimal value function does not immediately yield an approximate-optimal policy of comparable quality (see e.g. \cite{bertsekas2013abstract}) and it is was previously unclear how to combine these methods to improve upon previous known bounds for computing an approximate policy.
To achieve our policy computation algorithm we therefore open up both the algorithms and the analysis in \cite{azar2013minimax} and \cite{sidford2018variance}, combining them in nontrivial ways. Our proofs leverage techniques ranging from standard probabilistic analysis tools such as Hoeffding and Bernstein inequalities, to optimization techniques such as variance reduction, to properties specific to MDPs such as the Bellman fixed-point recursion for expectation and variance of the optimal value vector, and monotonicity of value iteration. %For a more in depth overview of our approach, see Section~\ref{sec:overview}.

Finally, we extend our method to finite-horizon MDPs, which are also occurred frequently in real applications. We show that the number of samples needed by this algorithm is 
$
\wt{O}(H^{3}|\cS||\cA| \epsilon^{-2}),
$
in order to obtain an $\epsilon$-optimal policy for $H$-horizon MDP (see Appendix~\ref{sec:finite_horizon}). We also show that the preceding sample complexity is optimal up to logarithmic factors by providing a matching lower bound. We hope this work ultimately opens the door for future practical and theoretical work on solving MDPs and efficient RL more broadly.

%The rest of this paper is organized as follows. In Section~\ref{sec:prelim} we introduce the notation and fundamental concepts we will use throughout the paper. In Section~\ref{sec:overview} we provide a more technical overview of our approach. In Section~\ref{sec:alg_policy} we provide and analyze our approximate policy computation algorithm. In Section~\ref{sec:prev_work} we formally compare our results to previous work.Finally, in Appendix~\ref{sec:alg_value} we formally show how to combine \cite{azar2013minimax} and \cite{sidford2018variance} to provide the first nearly time optimal value computation algorithm, in Appendix~\ref{sec:var_bounds} we provide various mathematical tools to analyze the variance in our QVI algorithm, and in Appendix~\ref{sec:lower bound} we provide the nearly matching lower bound on sample complexity as an easy extension of \cite{azar2013minimax}. In Appendix~\ref{sec:finite_horizon} we present our algorithm for finite horizon MDPs.

\def\cS{\mathcal{S}}
\def\cA{\mathcal{A}}
\def\S{|\mathcal{S}|}
\def\A{|\mathcal{A}|}
\def\cO{\mathcal{O}}
\def\tO{\tilde\cO}

\section{Comparison to Previous Work}
\label{sec:prev_work}
%\mw{To Carrie: could you please add the complexity table into this section?}

\begin{table*}[h]
\begin{center}\vspace{-.5em}
	{\small
	%\begin{tabular}{|c|c|c|}
   \begin{tabular}{|>{\centering\arraybackslash}m{4cm}|>{\centering\arraybackslash}m{5cm}|>{\centering\arraybackslash}m{3cm}|}
		\hline
		\textbf{Algorithm} & \textbf{Sample Complexity} & \textbf{References} \\ 
		\hline
		& &
		\\[-1em]
		Phased Q-Learning& $\wt{O} (C\frac{\S \A}{(1-\gamma)^7\epsilon^2}) $  & \cite{kearns1999finite}
		\\
		\hline
		& &
		\\[-1em]
		Empirical QVI& $\wt{O} (\frac{\S \A}{(1-\gamma)^5\epsilon^2})$ \footnotemark& \cite{azar2013minimax}
		\\
		\hline
		&&
		\\[-1em]
		%\multirow{2}{*}{Empirical QVI}  
		Empirical QVI & $\wt{O} \big(\frac{\S \A}{(1-\gamma)^3\epsilon^2}\big)$ if $\epsilon = \wt{O}\big(\frac1{\sqrt{(1-\gamma)\S}}\big) $  & %\multirow{2}{*}
		{ \cite{azar2013minimax}}
		%\\
		%\cline{2-2}
		%& &
		%\\[-1em]
		% & $\S^2\A \frac{\log(1/(1-\gamma)\epsilon)}{1-\gamma}$ if $\epsilon = \Omega(\frac1{\sqrt{(1-\gamma)\S}}) $  &
		%\\[-1em]
		%& &
		\\
		\hline
		Randomized Primal-Dual Method& $\wt{O} (C\frac{\S \A}{(1-\gamma)^4\epsilon^2}) $  & \cite{wang2017randomized}\\
		\hline
		Sublinear Randomized 
		Value Iteration & $\wt{O}  \left( \frac{|\cS| |\cA|}{(1-\gamma)^4 \epsilon^2}\right)$ & \cite{sidford2018variance}
		\\
		\hline
		& &
		\\[-1em]
		Sublinear Randomized QVI & $\wt{O}  \left( \frac{|\cS| |\cA|}{(1-\gamma)^3 \epsilon^2}\right)$ & This Paper
		\\[-1em]
		& &
		\\
		\hline
		%     Algorithm 1 for ergodic DMDP & $\tO (\S^2\A + \frac{\S \A}{(1-\gamma)^2\epsilon^2}) $ & Corollary 2\\
		%      \hline		
	\end{tabular}
	}\vspace{-.5em}
\end{center}
	%\vspace{-1.5em}
\caption{ 
	\small
	\textbf{Sample Complexity to Compute $\epsilon$-Approximate Policies Using the Generative Sampling Model}: Here $\S$ is the number of states, $\A$ is the number of actions per state, $\gamma \in(0,1)$ is the discount factor, and $C$ is an upper bound on the ergodicity. Rewards are bounded between 0 and 1. 
	\label{table-sample}}
	\label{tab:literature_runtime_apx}
	\vspace{-1em}
\end{table*}
\footnotetext{Although not explicitly stated, an immediate derivation shows that obtaining an $\epsilon$-optimal policy in \cite{azar2013minimax} requires $O(|S||A|(1-\gamma)^{-5}\epsilon^{-2})$ samples.}

There exists a large body of literature on MDPs and RL  (see e.g. \cite{kakade2003sample, strehl2009reinforcement, kalathil2014empirical, dann2015sample} and reference therein). 
The classical MDP problem is to compute an optimal policy exactly or approximately, when the full MDP model is given as input. For a survey on existing complexity results when the full MDP model is given, see Appendix~\ref{sec:full model compare}.

Despite the aforementioned results of \cite{kakade2003sample, azar2013minimax, sidford2018variance},
there exists only a handful of additional RL methods that achieve a small sample complexity and a small run-time complexity at the same time for computing an $\epsilon$-optimal policy.
A classical result is the phased Q-learning method by \cite{kearns1999finite}, which takes samples from the generative model and runs a randomized value iteration. The phased Q-learning method finds an $\epsilon$-optimal policy using $\cO(|\cS| |\cA| \epsilon^{-2} / \text{poly}(1-\gamma))$ samples/updates, where each update uses  $\wt{O}(1)$ run time.\footnote{The dependence on $(1-\gamma)$ in \cite{kearns1999finite} is not stated explicitly but we believe basic calculations yield  $O(1/(1-\gamma)^7)$.}
Another work \cite{wang2017randomized} gave a randomized mirror-prox method that applies to a special Bellman saddle point formulation of the DMDP. They achieve a total runtime of $\wt{O} (|\cS|^3 |\cA| \epsilon^{-2} (1-\gamma)^{-6} )$ for the general DMDP and $\wt{O} (C |\cS| |\cA| \epsilon^{-2} (1-\gamma)^{-4}) $ for DMDPs that are ergodic under all possible policies, where $C$ is a problem-specific ergodicity measure.  
A recent closely related work is \cite{sidford2018variance} which gave a variance-reduced randomized value iteration that works with the generative model and finds an $\epsilon$-approximate policy in sample size/run time $\wt{O} (|\cS| |\cA| \epsilon^{-2} (1-\gamma)^{-4}) $, without requiring any ergodicity assumption. 
%\lin{I feel the following commented graph is redundant from the intro. I have removed them. It is still in the source code. You can recover it if you feel appropriate.} \sidford{I think it is good to include some of this because we haven't talked about the sublinear time distinction. I editted a bit but brought some of it back. Let me know if you think we should do differently or remove the comments. Thanks!}

Finally, in the case where $\epsilon = {O}\Big(1/{\sqrt{(1-\gamma)^{-1}\S}}\Big)$,  \cite{azar2013minimax} showed that the solution obtained by performing exact PI on the empirical MDP model provides not only an $\epsilon$-optimal value but also an $\epsilon$-optimal policy. In this case, 
the number of samples is $\wt{O} ( \S \A  (1-\gamma)^{-3} \epsilon^{-2})$ and matches the sample complexity lower bound.
%the overall runtime is dominated by the time needed to construct the empirical model, which is $\wt{O} ( \S \A  (1-\gamma)^{-3} \epsilon^{-2})$ and matches the sample complexity lower bound, so it is also runtime optimal. 
Although this sample complexity is optimal, it requires solving the empirical MDP exactly (see Appendix~\ref{sec:alg_value}), and is no longer sublinear in the size of the MDP model because of the very small approximation error $\epsilon = {O}(1/{\sqrt{(1-\gamma)\S}})$.
%\cite{azar2013minimax} does not explicit state any running time results.
%Moreover, exact PI over the empirical model requires solving linear system of dimension $|\cS|$ and thus does not provide optimal running time.
%\sidford{Should we have added a line to the table for this reference with $1/(1 - \gamma)^5$ as this is what mapping this to a policy gets trivially?}
%\sidford{Is a tilde needed here for consistency with paragraph start?} 
%In the setting where one is interested in  approximation error of $\epsilon\geq \wt{\Omega}(1/{\sqrt{(1-\gamma)\S}})$, there remained a gap between the best known runtime and the theoretical lower bound. Thus the contribution of this work is to close this gap and provide a method that achieves both sample-optimal complexity and sublinear runtime for finding $\epsilon$-optimal policies for arbitrary fixed values of $\epsilon.$ 
See Table~\ref{table-sample} for a list of comparable sample complexity results for solving MDP based on the generative model.%\mw{Please check the preceding paragraph again} 

\section{Preliminaries}
\label{sec:prelim}

%\sidford{Note the use of $\valuespace$ and $\policyspace$. I think this looks nicer than putting cardinalities in the superscripts.}
We use calligraphy upper case letters for sets or operators, e.g., $\cS$, $\cA$ and $\cT$.
%We denote that $S=|\cS|$, and $A=|\cA|$.
We use bold small case letters for vectors, e.g., $\bv, \br$.
We denote $\bv_{s}$ or $\bv(s)$ as the $s$-th entry of vector $\bv$.
We denote matrix as bold upper case letters, e.g., $\bP$.
We denote constants as normal upper case letters, e.g., $M$.
For a vector $\bv\in \RR^{\cN}$ for index set $\cN$, we denote $\sqrt{\bv}$, $|\bv|$, and $\bv^2$ vectors in $\RR^{\cN}$ with $\sqrt{\cdot}$, $|\cdot|$, and $(\cdot)^2$ acting coordinate-wise.
For two vectors $\bv, \bu\in \RR^{\cN}$, we denote by $\bv\le \bu$ as coordinate-wise comparison, i.e., $\forall i\in \cN: \bv(i)\le \bu(i)$. The same definition are defined to relations $\le$, $<$ and $>$.

We describe a DMDP by the tuple $(\cS, \cA, \bP, \br, \gamma)$, where $\cS$ is a finite state space, $\cA$ is a finite action space, $\bP\in\RR^{\cS\times \cA\times \cS}$ %\sidford{Personally I'm a fan of not using the cardinality symbol in defining vectors spaces, i.e. writing $\mathbb{R}^{\cS}$ rather than $\mathbb{R}^{\cS}$.}
 is the state-action-state transition matrix, $\br\in\RR^{\cS\times \cA}$ is the state-action reward vector, and $\gamma \in (0, 1)$ is a discount factor. We use $\bP_{s, a}(s')$ to denote the probability of going to state $s'$ from state $s$ when taking action $a$. We also identify each $\bP_{s,a}$ as a vector in $\RR^{S}$. We use $\br_{s,a}$ to denote the reward obtained from taking action $a \in \actions$ at state $s \in \states$ and assume  $\br \in [0, 1]^{\cS\times \cA}$.\footnote{A general $\br\in\RR^{\cS\times \cA}$ can always be reduced to this case by shifting and scaling.}
For a vector $\bv\in \RR^{\cS}$, we denote $\bP\bv\in \RR^{\cS\times \cA}$ as $(\bP\bv)_{s,a} = \bP_{s,a}^\top \bv$.
A policy $\pi:\cS\to\cA$ maps each state to an action. The objective of MDP is to find the optimal policy $\pi^*$ that maximizes the expectation of the cumulative sum of discounted rewards.

%\poly\log(\tau_s)$. Note that if the oracle needs time $\tau_s$, we only need to change the every running time result in the rest of the paper by a factor of $\tau_s\poly\log(\tau_s)$.
%This is a natural assumption and can be related to varies models.
%For instance,  under standard arithmetic models of computation, 
%such an $O(\log|\cS|)$ sampling time oracle can be implemented by preprocessing the DMDP in linear, i.e. $O(|S|^2 |A|)$ time.
%\sidford{I removed some of the stuff that came previously, I wonder if it might be a good idea to find the reference for $O(1)$ expected time for sampling.}

%that for any state $s \in \states$ and action $a \in \actions$  in expected $O(1)$ time we can sample $s' \in \states$ independently at random with $\Pr[s'= s''] = \bP_{s,a}(s'')$.
%This is a natural assumption regarding a DMDP as under standard arithmetic models of computation can be satisfied by preprocessing the DMDP in linear, i.e. $O(|S|^2 |A|)$ time.\footnote{As discussed in the introduction if instead sampling required $O(\log \S)$ time, which is easily achieved with $\otilde(|\cS|^2 |\cA|)$ preprocessing, this would increase our running times by only a multiplicative $O(\log \S)$ which would be hidden by the $\otilde(\cdot)$ notation, leaving running times unaffected. 

%\mw{$\otilde(\S^2 \A)$ because arranging arrays into a tree takes $n\log n$. Can we do better?}
%For further discussion of sampling schemes, see \cite{wang2017randomized}. 

In the remainder of this section we give definitions for several prominent concepts in MDP analysis that we use throughout the paper. 

%\mw{Use a more special notation for the value operator? $T$ is also used for the iteration number}

\begin{definition}[Bellman Value Operator]
For a given DMDP the \emph{value operator} $\cT : \valuespace \mapsto \valuespace$ is defined for all $u \in \valuespace$ and $s \in \states$ by
%\begin{equation}
$\cT(\bu)_{s} = \max_{a \in \actions} [ \br_{a}(s) + \gamma \cdot \bP_{s,a}^\top \bv ], \label{eq:value_operator}
$
%\end{equation}
and we let $\vopt$ denote the {\emph{value of the optimal policy $\pi^*$}}, which is the unique vector such that $\cT(\vopt) = \vopt$. 
\end{definition}

\begin{definition}[Policy] We call any vector $\pi \in \policyspace$ a \emph{policy} and say that the action prescribed by policy $\pi$ 
%\sidford{why aren't policies in bold?}\lin{It is a map, not a vector?}\sidford{What's the different in this finite case? We write $\pi \in \actions^\states$ which is vector notation, right? I agree $\pi$ is not real valued, but it is a vector :P}  
to be taken at state $s \in \states$ is $\pi_s$. We let $\cT_\pi : \valuespace \mapsto \valuespace$ denote the \emph{value operator associated with $\pi$} defined for all $u \in \valuespace$ and $s \in \states$ by
$
\cT_\pi(\bu)_s = \br_{s, \pi(s)} + \gamma \cdot \bP_{s, \pi(s)}^\top \bu ~,  
$ 
and we let $\bv^\pi$ denote the \emph{values of policy $\pi$}, which is the unique vector such that $\cT_\pi(\bv^\pi) = \bv^\pi$. 
\end{definition}
\vspace{-2mm}
Note that $\cT_\pi$ can be viewed as the value operator for the modified MDP where the only available action from each state is given by the policy $\pi$. Note that this modified MDP is essentially just an uncontrolled Markov Chain, i.e. there are no action choices that can be made. 

\begin{definition}[$\epsilon$-optimal value and policy] We say values $\bu \in \valuespace$ are \emph{$\epsilon$-optimal} if $\| \vopt - \bu\|_{\infty} \leq \epsilon$ and policy $\pi \in \policyspace$ is \emph{$\epsilon$-optimal} if $\| \vopt - \bv^{\pi}\|_{\infty} \leq \epsilon$, i.e. the values of $\pi$ are $\epsilon$-optimal.
\end{definition}

\begin{definition}[Q-function] 
	For any policy $\pi$, we define the Q-function of a MDP with respect to $\pi$ as a vector $\bQ\in \RR^{\cS\times\cA}$ such that
	$
		\bQ^{\pi}(s,a) = \br(s, a) + \gamma\bP_{s, a}^\top \bv^{\pi}.
	$
	The optimal $Q$-function is defined as $\bQ^{*} = \bQ^{\pi^*}$.
%	The optimal $Q$-function is defined as
%	\[
%		\bQ^{*}(s,a) =\br(s, a) + \gamma\bP_{s, a}^\top \bv^{*}.
%	\]
	We call any vector $\bQ\in \RR^{\cS\times\cA}$ a Q-function even though it may not relate to a policy or a value vector and define 
$\bv(\bQ)\in \RR^{\cS}$ and $\pi(\bQ)\in \cA^{\cS}$ as the value and policy implied by $\bQ$, by
\[
\forall s\in \cS: \bv(\bQ)(s) = \max_{a\in \cA} \bQ(s,a) \quad\text{and}\quad \pi(\bQ)(s) = \arg\max_{a\in \cA} \bQ(s,a).
\]
For a policy $\pi$, let $\bP^{\pi}\bQ\in\RR^{\cS\times\cA}$ be defined as $(\bP^{\pi}\bQ)(s,a) = \sum_{s'\in\cS}\bP_{s,a}(s')\bQ(s',\pi(s'))$.
\end{definition}
%\mw{Is the last paragraph necessary ? seems repetitive}

\section{Technique Overview}
\label{sec:overview}

In this section we provide a more detailed and technical overview of our approach. 
At a high level, our algorithm shares a similar framework as the variance reduction algorithm presented in \cite{sidford2018variance}.
This algorithm used two crucial algorithmic techniques, which are also critical in this paper.
We call these techniques as the \emph{monotonicity} technique and the \emph{variance reduction} technique.
Our algorithm and the results of this paper can be viewed as an advanced, non-trivial integration of these two methods, augmented with a third technique which we refer to as a \emph{total-variation} technique which was discovered in several papers \cite{munos1999variable, lattimore2012pac, azar2013minimax}.
%However these papers do not  give nearly optimal running time/sample complexity bounds for solving the MDP in terms of parameters $|\cS|, |\cA|, \epsilon$ and $(1-\gamma)$.
%The non-trivial 
%combination of these three techniques yields the main results of this paper. 
In the remainder of this section we give an overview of these techniques and through this, explain our algorithm.

\paragraph{The Monotonicity Technique}
Recall that the classic value iteration algorithm for solving a MDP repeatedly applies the following rule 
%\sidford{I would remove the small and the bracket everywhere.}
\begin{align}
\label{eqn:value iteration}
\bv^{(i)}(s)\gets \max_{a} (r(s,a) + \gamma\bP_{s,a}^\top \bv^{(i-1)}).
\end{align}
%where $\bv^{(0)}={\bf0}$.\sidford{Does value iteration have to start from $0$? I might just remove the last phrase about $v^{0}$?}
A greedy policy $\pi^{(i)}$ can be obtained at each iteration $i$ by 
\begin{align}
\label{eqn:greedy policy}
\forall s: \pi^{(i)}(s)\gets \argmax_{a} (r(s,a) + \gamma\bP_{s,a}^\top \bv^{(i)}).
\end{align}
For any $u>0$, it can be shown that if one can approximate $\bv^{(i)}(s)$ with $\wh{\bv}^{(i)}(s)$ such that $\|\wh{\bv}^{(i)}-\bv^{(i)}\|_{\infty}\le (1-\gamma)u$ %\sidford{Is there some reason that $u$ and not $\epsilon$ was used in this section?}\lin{$\epsilon$ is the final precision, $u$ is the intermediate one. We need a way to distinguish the two.} 
%\sidford{should say all $s$ or write $\ell_\infty$ norm?} 
and run the above value iteration algorithm using these  approximated values, then after $\Theta((1-\gamma)^{-1}\log[u^{-1}(1-\gamma)^{-1}])$ iterations, the final iteration gives an value function that is $u$-optimal (\cite{bertsekas2013abstract}). %\mw{Be a bit more precise here, cause approx VI doesnt have to converge}\lin{done}
However, %as mentioned in Section~\ref{sec:fast value}\sidford{where is this?}, 
a $u$-optimal value function only yields a $u/(1-\gamma)$-optimal greedy policy (in the worst case), even if \eqref{eqn:greedy policy} is precisely computed.
To get around this additional loss, a monotone-VI algorithm was proposed in \cite{sidford2018variance} as follows. %\mw{who proposed it? shall we cite the soda paper here? or just explain it informally?}\lin{done}
At each iteration, this algorithm maintains not only an approximated value $\bv^{(i)}$ but also a policy $\pi^{(i)}$. The key for improvement is to keep values as a lower bound of the value of the policy on a set of sample paths with high probability. In particular, the following {\it monotonicity condition} was maintained with high probability
%{\small
\[
\bv^{(i)}\le \cT_{\pi^{(i)}} (\bv^{(i)}) ~.
\] 
%}
By the monotonicity of the Bellman's operator, the above equation guarantees that $\bv^{(i)}\le \bv^{\pi^{(i)}}$.
If this condition is satisfied, then, if after $R$ iterations of approximate value iteration 
we obtain an value $\wh{\bv}^{(R)}$ that is $u$-optimal then 
we also obtain a policy $\pi^{(R)}$ which by the monotonicity condition and the monotonicity of the Bellman operator $\cT_{\pi^{(R)}}$ yields
%{\small
\[
\bv^{(R)}
\le \cT_{\pi^{(R)}}(\bv^{(R)})
\le \cT_{\pi^{(R)}}^2 (\bv^{(R)})  
\le \ldots 
\le \cT_{\pi^{(R)}}^{\infty} (\bv^{(R)})
=\bv^{\pi^{(R)}} \le \bv^{*}.
\]
%}
and therefore this $\pi^{(R)}$ is an $u$-optimal policy. Ultimately, this technique avoids the standard loss of a $(1-\gamma)^{-1}$ factor when converting values to policies.

\paragraph{The Variance Reduction Technique}
Suppose now that we provide an algorithm that maintains the monotonicity condition using random samples from $\bP_{s,a}$ to approximately compute \eqref{eqn:value iteration}. 
%\sidford{Writing in this section is a little confusing, would make clear that here we are still describing previous paper and not this new paper.}\lin{How about we put a sentence in the beginning of this section: Following from Sidford et al ...}
Further, suppose we want to obtain a new value function and policy that is at least $(u/2)$-optimal.
In order to obtain the desired accuracy,
we need to approximate $\bP_{s,a}^{\top}\bv^{(i)}$ up to error at most $(1-\gamma)u/2$.
Since $\|\bv^{(i)}\|_{\infty}\le (1-\gamma)^{-1}$, by Hoeffding bound, $\wt{O}((1-\gamma)^{-4}u^{-2})$ samples suffices. Note 
%\sidford{Recall from where, is this a new claim?}\lin{Changed to note} 
that the number of samples also determines the computation time and therefore each iteration takes $\wt{O}((1-\gamma)^{-4}u^{-2} |\cS||\cA|)$ samples/computation time and $\wt{O}((1-\gamma)^{-1})$ iterations for the value iteration to converge. Overall, this yields a sample/computation complexity of $\wt{O}((1-\gamma)^{-5}u^{-2} |\cS||\cA|)$.
To reduce the $(1-\gamma)^{-5}$ dependence, \cite{sidford2018variance} uses properties of the input  (and the initialization) vectors: $\|\bv^{(0)} - \bv^{*}\|_{\infty}\le u$ and rewrites value iteration \eqref{eqn:value iteration} as follows
\begin{align}
\label{eqn:value iteration variance reduction}
\bv^{(i)}(s)\gets \max_{a} \bigr[r(s,a) + \bP_{s,a}^\top (\bv^{(i-1)} - \bv^{(0)}) + \bP_{s,a}^\top\bv^{(0)}\bigr],
\end{align}
Notice that $\bP_{s,a}^\top\bv^{(0)}$ is shared over all iterations and we can approximate it up to error $(1-\gamma)u/4$ using only $\wt{O}((1-\gamma)^{-4}u^{-2})$ samples.
For every iteration, we have $\|\bv^{(i-1)} - \bv^{(0)}\|_{\infty}\le u$ (recall that we demand the monotonicity is satisfied at each iteration).
Hence $\bP_{s,a}^\top (\bv^{(i-1)} - \bv^{(0)})$ can be approximated up to error $(1-\gamma)u/4$ using only $\wt{O}((1-\gamma)^{-2})$ samples (note that there is no $u$-dependence here).
By this technique, over $\wt{O}((1-\gamma)^{-1})$ iterations only
$
	\wt{O}((1-\gamma)^{-4}u^{-2} + (1-\gamma)^{-3})
$
samples/computation per state action pair are needed, i.e. there is a $(1-\gamma)$ improvement.

\paragraph{The Total-Variance Technique}
By combining the monotonicity technique and variance reduction technique, one can obtain a $\wt{O}((1-\gamma)^{-4})$ sample/running time complexity (per state-action pair) on computing a policy; this was one of the results \cite{sidford2018variance}. 
However, there is a gap between this bound and the best known lower bound of $\wt{\Omega}[|S| |A| \epsilon^{-2} (1-\gamma)^{-3}]$ \cite{azar2013minimax}.
Here we show how to remove the last $(1-\gamma)$ factor by better exploiting the structure of the MDP. In \cite{sidford2018variance} the update error in each iteration was set to be at most $(1-\gamma) u/2$ to compensate for error accumulation through a horizon of length $(1-\gamma)^{-1}$ (i.e., the accumulated error is sum of the estimation error at each iteration). To improve we show how to leverage previous work to show that the true error accumulation is much less. To see this, let us now switch to Bernstein inequality.
Suppose we would like to estimate the value function of some policy $\pi$.  The estimation error vector of the value function is upper bounded by $\wt{O}(\sqrt{\bsigma_{\pi}/m})$, where $\bsigma_{\pi}(s) = \var_{s'\sim \bP_{s, \pi(s)}}(\bv^{\pi}(s'))$ %\sidford{should that be lower case $v$?} 
denotes the variance % of the total return 
of the value of the next state if starting from state $s$ by playing policy $\pi$, and $m$ is the number of samples collected per state-action pair.
The accumulated error due to estimating value functions can be shown to obey the following inequality (upper to logarithmic factors) 
%\sidford{is the LHS exactly the error, if so, I would add an equality to the left to make this clear}
%{\small
\[
\text{accumated error}\propto\sum_{i=0}^{\infty} \gamma^{i}\bP_{\pi}^i\sqrt{\bsigma_{\pi}/m} \le c_1 
\left(\frac{1}{1-\gamma}\sum_{i=0}^{\infty}\gamma^{2i}\bP_{\pi}^i\bsigma_{\pi}/m\right)^{1/2},
\]
%}
where $c_1$ is a constant and the inequality follows from a Cauchy-Swartz-like inequality.
According to the \emph{law of total variance}, for any given policy $\pi$ (in particular, the optimal policy $\pi^*$) and initial state $s$, the expected sum of variance of the tail sums of rewards, $\sum\gamma^{2i}\bP_{\pi}^i\bsigma_{\pi}$, is exactly the variance of the total return by playing the policy $\pi$. This observation was previously used in the analysis of \cite{munos1999variable, lattimore2012pac,azar2013minimax}.
Since the upper bound on the total return is $(1-\gamma)^{-1}$, it can be shown that $\sum_i\gamma^{2i}\bP_{\pi}^i\bsigma_{\pi}\le (1-\gamma)^{-2}\cdot \one$ %\mw{is this a scalar or vector?} 
and therefore the total error accumulation is 
$\sqrt{(1-\gamma)^{-3}/m}$. Thus picking $m\approx(1-\gamma)^{-3}\epsilon^{-2}$ is sufficient to control the accumulated error (instead of $(1-\gamma)^{-4}$). To analyze our algorithm, we will apply the above inequality to the optimal policy $\pi^*$ to obtain our final error bound.

\paragraph{Putting it All Together} In the next section we show how to combine these three techniques into one algorithm and make them work seamlessly. % (in Algorithm~\ref{alg-halfErr}).
%\sidford{Lin if you have the energy to add a few sentences describing the code in English that would be great.} 
In particular, we provide and analyze how to combine these techniques into an
 Algorithm~\ref{alg-halfErr} which can be used to at least halve the error of a current policy. 
Applying this routine a logarithmic number of time then yields our desired bounds.
In the input of the algorithm, we demand the input value  $\bv^{(0)}$ and $\pi^{(0)}$ satisfies the required monotonicity requirement, i.e., $\bv^{(0)}\le \cT_{\pi^{(0)}} (\bv^{(0)})$ (in the first iteration, the zero vector $\bf{0}$ and an arbitrary policy $\pi$ satisfies the requirement). 
We then pick a set of samples to estimate $\bP \bv^{(0)}$ accurately with $\wt{O}((1-\gamma)^{-3}\epsilon^{-2})$ samples per state-action pair. 
The same set of samples is used to estimate the variance vector $\bsigma_{\bv^*}$.
These estimates serve as the initialization of the algorithm.
In each iteration $i$, we draw fresh new samples to compute estimate of $\bP (\bv^{(i)} - \bv^{(0)})$. 
The sum of the estimate of $\bP \bv^{(0)}$ and $\bP (\bv^{(i)} - \bv^{(0)})$ gives an estimate of $\bP \bv^{(i)}$. 
We then make the above estimates have one-sided error by shifting them according to their estimation errors (which is estimated from the Bernstein inequality).
These one-side error estimates allow us to preserve monotonicity, i.e., guarantees the new value is always improving on the entire sample path with high probability.
The estimate of $\bP \bv^{(i)}$ is plugged in to the Bellman's operator and gives us new value function, $\bv^{(i+1)}$ and policy $\pi^{(i+1)}$, satisfying the monotonicity and advancing accuracy.
Repeating the above procedure for the desired number of iterations completes the algorithm.

%\sidford{Thanks for tweaking pseudocode. I tweaked it a little more. Once you are happy with it, if you could do the analagous thing with the finite horizon code, that would be great.}

%{\small
	\begin{algorithm}[htb!]
		\caption{Variance-Reduced QVI\label{alg-halfErr}}
		\begin{algorithmic}[1]
			\State 
			\textbf{Input:} A sampling oracle for DMDP $\cM=(\cS, \cA, \br, \bP, \gamma)$
			\State
			\textbf{Input:} Upper bound on error $u\in[0,(1-\gamma)^{-1}]$ and error probability $\delta \in (0, 1)$
			\State 
			\textbf{Input:}
			Initial values $\bv^{(0)}$ and policy $\pi^{(0)}$ such that $\bv^{(0)}\le \cT_{\pi^{(0)}}\bv^{(0)}$, and $\bv^*-\bv^{(0)} \le u \one$;
			\State\textbf{Output:} $\bv, \pi$ such that $\bv \le \cT_{\pi}(\bv)$ and $\bv^*-\bv \le (u/2)\cdot\one$.
			\State
			\State\textbf{INITIALIZATION:} 
			\State Let $\beta\gets (1-\gamma)^{-1}$, and $R\gets\lceil c_1\beta\ln[\beta u^{-1}]\rceil$ for constant $c_1$; 
			\State Let $m_1 \gets{c_2\beta^3u^{-2}{\log(8|\cS||\cA|\delta^{-1})} }{}$ for constant $c_2$; 
			\State Let $m_2\gets {c_3\beta^{2}\log[2R|\cS||\cA|\delta^{-1}]}$ for constant $c_3$;
			\State Let $\alpha_1\gets{m_1}^{-1}{\log(8|\cS||\cA|\delta^{-1})}$;
			\State
			For each  $(s, a)\in \cS\times\cA$,
			sample independent samples $s_{s,a}^{(1)}, s_{s,a}^{(2)}, \ldots, s_{s,a}^{(m_1)}$ from $\bP_{s,a}$;
			\State
			Initialize $\bw=\wt{\bw} = \wh{\bsigma}=\bQ^{(0)} \gets {\bf0}_{\cS \times \cA}$, and $i\gets 0$;	
			\For{each $(s, a)\in \cS\times\cA$} 
			\State \emph{\textbackslash \textbackslash Compute empirical estimates of $\bP_{s,a}^{\top}\bv^{(0)}$ and $\bsigma_{\bv^{(0)}}(s,a)$}
			\State 
			Let $\wt{\bw}(s,a) \gets \frac{1}{m_1} %m_1^{-1}
			\sum_{j=1}^{m_1} \bv^{(0)}(s_{s,a}^{(j)})$ 
			\label{alg1: compute w1} 
			\State 
			%Let $\wh{\bsigma}(s,a)\gets \gamma^2\big[{m_1^{-1}}\sum_{j=1}^{m_1}(\bv^{(0)})^2(s_{s,a}^{(j)}) - \wt{\bw}^2(s,a)\big]$ 
			Let $\wh{\bsigma}(s,a)\gets \frac{1}{m_1} \sum_{j=1}^{m_1}(\bv^{(0)})^2(s_{s,a}^{(j)}) - \wt{\bw}^2(s,a)$  \label{alg1: compute w} 
			\State
			
			\State \emph{\textbackslash \textbackslash Shift the empirical estimate to have one-sided error and guarantee monotonicity } 
			\State 	
			$\bw(s, a) \gets \wt{\bw}(s,a) - \sqrt{2\alpha_1\wh{\bsigma}(s,a)} - 4\alpha_1^{3/4}\norm{\bv^{(0)}}_\infty - (2/3)\alpha_1\norm{\bv^{(0)}}_{\infty}$
			
			\State
			
			\State \emph{\textbackslash \textbackslash Compute coarse estimate of the  $Q$-function}
			% \big({\frac{2\wh{\bsigma}(s,a)\cdot\log{(8|\cS||\cA|\delta^{-1})}}{m_1}}\big)^{1/2} - \frac{2\norm{\bv^{(0)}}_\infty\cdot\log{(8|\cS||\cA|\delta^{-1})}}{m_1}$, and
			\State
			$\bQ^{(0)}(s,a) \gets \br(s,a) + \gamma \bw(s,a)$

			\EndFor
			\State
			%\State\textbf{Repeat:}
			\State\textbf{REPEAT:} \emph{\qquad\qquad\textbackslash \textbackslash successively improve}
			\For{$i=1$ to $R$}
			\State  \emph{\textbackslash \textbackslash Compute $\bg^{(i)}$ the estimate of $\bP \big[\bv^{(i)}  - \bv^{(0)}\big]$ with one-sided error}
			\State\label{alg: v1} Let ${\bv}^{(i)} \gets \bv( \bQ^{(i-1)})$, ${\pi}^{(i)}\gets \pi(\bQ^{(i-1)})$; \textbackslash \textbackslash \emph{let $\wt{\bv}^{(i)}\gets \bv^{(i)}, \wt{\pi}^{(i)}\gets {\pi}^{(i)}$ (for analysis)}; 
			\State\label{alg: v2} For each $s\in \cS$, if ${\bv}^{(i)}(s)\le \bv^{(i-1)}(s)$, then 
			$\bv^{(i)}(s)\gets \bv^{(i-1)}(s)$ and $\pi^{(i)}(s)\gets \pi^{(i-1)}(s)$;
			\State For each $(s, a)\in \cS\times\cA$,
			draw independent samples $\wt{s}_{s,a}^{(1)}, \wt{s}_{s,a}^{(2)}, \ldots, \wt{s}_{s,a}^{(m_2)}$ from $\bP_{s,a}$;
			\State  \label{alg1: compute g} Let $\bg^{(i)}(s,a)\gets {\frac{1}{m_2}} \sum_{j=1}^{m_2} \big[\bv^{(i)}(\wt{s}_{s,a}^{(j)}) - \bv^{(0)}(\wt{s}_{s,a}^{(j)}) \big]- (1-\gamma)u/8$; \\
\State  \emph{\textbackslash \textbackslash Improve $\bQ^{(i)}$}
			\State \label{alg: q-func} $\bQ^{(i)}\gets \br + \gamma\cdot[\bw+\bg^{(i)}]$;
			\EndFor
			\State \textbf{return} $\bv^{(R)}, \pi^{(R)}$.
		\end{algorithmic}
	\end{algorithm}
	\vspace{-2mm}
%}

\section{Algorithm and Analysis}
\label{sec:alg_policy}

In this section we provide and analyze our near sample/time optimal
 $\epsilon$-policy computation algorithm. As discussed in Section~\ref{sec:overview} our algorithm combines three main ideas: variance reduction, the monotone value/policy iteration, and the reduction of accumulated error via Bernstein inequality. These ingredients are used in the Algorithm~\ref{alg-halfErr} to provide a routine which halves the error of a given policy. We analyze this procedure in Section~\ref{sub:var_reduce} and use it to obtain our main result in Section~\ref{sub:halving}.

%As discussed, obtaining an $\epsilon$-optimal value function is not sufficient for obtaining an $\epsilon$-optimal policy (see e.g. \cite{bertsekas2013abstract}). 
%\sidford{Somewhere we should possibly give a more concrete example of this, i.e. where after having the approximate values and sampling to do one step of value operator, we don't get a comparable policy.}
%\lin{Added a citation.} \sidford{Thanks, but are we going to add the proof somewhere that can go from value to policy only losing a $1 - \gamma$ factor? If so, I would also just ref that.}
%In this section, we describe our algorithm that combines value iteration and policy iteration. 
%\sidford{I changed this to avoid calling what we are doing policy iteration, I think that is confusing and doesn't add much.}

%\sidford{How use policy iteration? I wouldn't call this policy iteration.}
%\lin{Since we moved the meta algorithm to the bottom, we should put some explanation here.}

\subsection{The Analysis of the Variance Reduced Algorithm}
\label{sub:var_reduce}
% Sidford: removed reminder for space savings
% We first recall some notations.
%For a $Q$-function $\bQ\in \RR^{\cS\times\cA}$, we recall the definition of 
% $\bv:=\bv(\bQ)\in \RR^{\cS}$ and $\pi: = \pi(\bQ)\in \cA^{\cS}$ as 
%$\max \bQ\in\RR^{\cS}$ as an $|\cS|$-dimensional vector and $\pi = \pi(\bQ)$%\arg\max \bQ$ 
%as a policy, where
%\[
%\forall s\in \cS: \bv(\bQ)(s) = \max_{a}\bQ(s,a) \quad\text{and}\quad \pi(s) = \arg\max_{a} \bQ(s,a). 
%\]

In this section we analyze Algorithm~\ref{alg-halfErr}, showing that each iteration of the algorithm approximately contracts towards the optimal value and policy and that ultimately the algorithm halves the error of the input value and policy with high probability. All proofs in this section are deferred to Appendix~\ref{sec:proof of main alg}.

We start with bounding the error of $\wt{\bw}$ and $\wh{\bsigma}$ defined in Line \ref{alg1: compute w1} and \ref{alg1: compute w} of Algorithm 1.
Notice that these are the empirical estimations of $\bP_{s,a}^{\top}\bv^{(0)}$ and $\bsigma_{\bv^{(0)}}(s,a)$.
\begin{lemma}[Empirical Estimation Error]
    \label{lemma: emprical mean}
	Let $\wt{\bw}$ and $\wh{\bsigma}$ be computed in Line \ref{alg1: compute w1} and \ref{alg1: compute w} of Algorithm \ref{alg-halfErr}.
	Recall that $\wt{\bw}$ and $\wh{\bsigma}$ are empirical estimates of $\bP\bv$ and $\bsigma_{\bv}=\bP\bv^2 - (\bP\bv)^2$ using $m_1$ samples per $(s,a)$ pair.
	%For each $(s,a)\in \cS\times \cA$, let $\bsigma_{\bv}(s, a)=\gamma^2[\bP_{s,a}^{\top}\bv^2 - (\bP_{s,a}^\top\bv)^2]$ be the scaled variance.
	With probability at least $1-\delta$, for $L \defeq \log(8|\cS||\cA|\delta^{-1})$, we have 
	\begin{align}
	\label{eqn:estimate pv}
	%\forall (s,a)\in \cS\times \cA:\quad	
    \big|{\wt{\bw} - \bP^\top\bv^{(0)}}\big|\le \sqrt{{2m_1^{-1}\bsigma_{\bv^{(0)}}\cdot{L}}} + {2(3m_1)^{-1}\norm{\bv^{(0)}}_{\infty} {L}} 
	\end{align}
and
	\begin{align}
		\label{eqn:estimate sigma}
		\forall (s,a)\in \cS\times \cA:\quad
		\big|\wh{\bsigma}(s,a) - \bsigma_{\bv^{(0)}}(s, a)\big| \le 4\norm{\bv^{(0)}}_{\infty}^2 \cdot \sqrt{{2m_1^{-1}{L}}}.
	\end{align}
\end{lemma}
The proof is a straightforward application of Bernstein's inequality and Hoeffding's inequality.

Next we show that the difference between $\sigma_{\bv^{(0)}}$ and $\sigma_{\bv^{*}}$ is also bounded. 

\begin{lemma}
\label{lemma: variance triangle}
Suppose $\norm{\bv-\bv^*}_{\infty}\le \epsilon$ for some $\epsilon > 0$, then 
$
\sqrt{\bsigma_{\bv}} \le \sqrt{\bsigma_{\bv^*}} + \epsilon\cdot\one.
$
\end{lemma}
Next we show that in Line~\ref{alg1: compute g}, the computed $\bg^{(i)}$ concentrates to and is an overestimate of $\bP[\bv^{(i)}  - \bv^{(0)}]$ with high probability.
\begin{lemma}
	\label{lemma:bounds on g}
	Let $\bg^{(i)}$ be the estimate of $\bP\big[\bv^{(i)}  - \bv^{(0)}\big]$ defined in Line~\ref{alg1: compute g} of Algorithm~\ref{alg-halfErr}. 
    %Let $\delta' = \delta/R$.\sidford{I would remove the definition of $\delta'$ and just propagate the $/R$ where needed, don't think it is worth the extra definition.}
 %for some integer $R> 0$.
	%Note that $\bg^{(i)}$ is a $[(1-\gamma)\epsilon/8]$-shifted empirical estimation with $m_2=128\cdot(1-\gamma)^{-2}\cdot\log(2|\cS||\cA|/\delta')$ samples per $(s,a)$,
	Then conditioning on the event that $\norm{\bv^{(i)} - \bv^{(0)}}_{\infty}\le 2u$, with probability at least $1-\delta/R$, 
	\vspace{-1mm}
	\[
		\bP  \big[\bv^{(i)}  - \bv^{(0)}\big] - \frac{(1-\gamma)u}{4}\cdot\one\le \bg^{(i)} \le \bP  \big[\bv^{(i)}  - \bv^{(0)}\big]
	\]
	provided appropriately chosen constants $c_1$, $c_2$, and $c_3$ in Algorithm~\ref{alg-halfErr}.
\end{lemma}

Now we present the key contraction lemma, in which we set the constants, $c_1, c_2, c_3$, in Algorithm~\ref{alg-halfErr} to be sufficiently large (e.g., $c_1 \ge 4, c_2\ge8192, c_3\ge128$).
%and show approximate contraction in each iteration. 
Note that these constants only need to be sufficiently large so that the concentration inequalities hold.
\begin{lemma}
	\label{lemma:induction lemma}
Let $\bQ^{(i)}$ be the estimated $Q$-function of $\bv^{(i)}$ in Line~\ref{alg: q-func} of Algorithm 1.
%Let $\bQ^* = \br + \gamma \bP \bv^*$ be the optimal $Q$-function of the DMDP.
Let $\pi^{(i)}$ and $\bv^{(i)}$ be estimated in iteration $i$, as defined in Line~\ref{alg: v1} and \ref{alg: v2}.
%Let $\pi^*$ be an optimal policy for the DMDP.
%\sidford{This should be defined in preliminaries and not here, right?.} 
%For a policy $\pi$, let $\bP^{\pi}\bQ\in\RR^{\cS\times\cA}$ be defined as $(\bP^{\pi}\bQ)(s,a) = \sum_{s'\in\cS}\bP_{s,a}(s')\bQ(s',\pi(s'))$.
Then, with probability at least $1- 2\delta$, for all
%\sidford{I would put parens around the input to value operator everywhere to avoid confusion.}\lin{Done} 
$1\le i \le R$, 
\[
\bv^{(i-1)}\le \bv^{(i)}\le \cT_{\pi^{(i)}}[\bv^{(i)}],\quad \bQ^{(i)}\le \br+\gamma\bP \bv^{(i)}, %\] 
\quad\text{and}\quad 
\bQ^* - \bQ^{(i)} \le \gamma \bP^{\pi^*}\big[\bQ^* - \bQ^{(i-1)}\big] 
 +\bxi, %2\sqrt{{2\alpha_1\gamma^{-2}\wh{\sigma}_{\bv^{*}}}} + 2\rbr{ \sqrt{2\alpha_1}\epsilon + 8\alpha_1^{3/4}\norm{\bv^{(0)}}_\infty + ({2}/{3})\alpha_1\norm{\bv^{(0)}}_{\infty} + (1-\gamma)\epsilon/8}\cdot\one,
\]
%\vspace{-1mm}
where for $\alpha_1=m_1^{-1}L< 1$ the error vector $\bxi$ satisfies %\sidford{add "$\bxi$" satisfies?}
\[
{\bf0}\le \bxi\le C\sqrt{\alpha_1\bsigma_{\bv^*}} +\Big[{(1-\gamma)u}/C+ C\alpha_1^{3/4}\norm{\bv^{(0)}}_\infty\Big]\cdot\one,
\]
for some sufficiently large constant $C\ge 8$.
%\sidford{Why not use the same $L$ as before and just write the $m_1$ inline when needed?} 
%\sidford{Also any reason not to merge the coeffecients on the one vector in the writing?}\lin{Did you mean each term? I think it may be good for the reader to know all the lower order terms in terms of $\alpha$? I am fine to merge them though.} \sidford{I agree, I just meant writing the following equivalent expression to save space. 
%\[
%{\bf0} 
%\le 
%2\sqrt{2\alpha_1\bsigma_{\bv^*}}
%+
%\left(
%{(1-\gamma)u}/{8}
%+ 
%2\sqrt{2\alpha_1}u
%+
%16\alpha_1^{3/4}\norm{\bv^{(0)}}_\infty
%+
%({4}/{3})\alpha_1 \norm{\bv^{(0)}}_{\infty}
%\right)\cdot \one
%,
%\]
%Up to you which to do.}
\end{lemma}

Using the previous lemmas we can prove the guarantees of Algorithm~\ref{alg-halfErr}.

%\sidford{Remind me, why proposition instead of theorem?} \lin{This is not the main theorem, so put it as a proposition.}\mw{It does feel weir to have a proposition here. Can it be moved later to the analysis?}

\begin{proposition}
	\label{thm:alg 1}
	%\mw{Proposition numbering seems weird}
	%Suppose we are able to sample a state from each $\bP_{s,a}$ with time $\tau_s$.
	%Let $\beta=(1-\gamma)^{-1}$,  $R=\lceil c_1\beta\ln[\beta u^{-1}]\rceil,
	%m_1= c_2\gamma^{-2}\beta^3u^{-2}\cdot\log(8|\cS||\cA|\delta^{-1})$ and $ m_2= c_3\beta^2\cdot\log[2R|\cS||\cA|\delta^{-1}]$ for some constants $c_1, c_2$ and $c_3$.
	On an input value vector $\bv^{(0)}$, policy $\pi^{(0)}$, and parameters $u\in(0, (1-\gamma)^{-1}], \delta\in(0,1)$ such that $\bv^{(0)}\le \cT_{\pi^{(0)}}[\bv^{(0)}]$, and $\bv^*-\bv^{(0)} \le u \one$,
	Algorithm~\ref{alg-halfErr} halts in time
	\[
	O\bigg[\bigg(u^{-2}
	+
	\log\frac{1}{(1-\gamma)u}\bigg)\cdot \frac{|\cS||\cA|\cdot \log(|\cS||\cA|\delta^{-1})}{(1-\gamma)^{3}}
	\bigg]
	%O[\gamma^{-2}u^{-2}\beta\cdot|\cS||\cA|\cdot \log(|\cS||\cA\delta^{-1}\beta u^{-1})]
	\]
	and outputs values $\bv$ and policy $\pi$ such that 
%	\begin{align*}
$	\bv \le \cT_{\pi}(\bv)$
%\quad\text{and}
 and 
$ \bv^*-\bv \le (u/2) \one$
%	\end{align*} 
with probability at least $1-\delta$, provided appropriately chosen constants, $c_1, c_2, c_3$.
\end{proposition}

We prove this proposition by iteratively applying Lemma~\ref{lemma:induction lemma}. Suppose 
$\bv^{(R)}$ is the output of the algorithm, after $R$ iterations. We show
$
\bv^* - \bv^{(R)} \le\gamma^{R-1}\bP^{\pi^*}\big[\bQ^* - \bQ^{0}\big] + (\bI-\gamma\bP^{\pi^*})^{-1}\bxi. %\le \frac{{u}}{2}\cdot\one.
$
Notice that $(\bI-\gamma\bP^{\pi^*})^{-1}\bxi$ is related to $(\bI-\gamma\bP^{\pi^*})^{-1}\sqrt{\bsigma_{\bv^*}}$.
We then apply the variance analytical tools presented in Section~\ref{sec:var_bounds} to show that 
$(\bI-\gamma\bP^{\pi^*})^{-1}\bxi \le (u/4) \one$ when setting the constants properly in Algorithm~\ref{alg-halfErr}.
We refer this technique as the \emph{total-variance technique}, since $\|(\bI-\gamma\bP^{\pi^*})^{-1}\sqrt{\bsigma_{\bv^*}}\|_{\infty}^2\le O[(1-\gamma)^{-3}]$ 
%\sidford{Should the norm be infinity norm?}
instead of a na\"ive bound of $(1-\gamma)^{-4}$. 
We complete the proof by choosing $R=\wt{\Theta}((1-\gamma)^{-1}\log(u^{-1}))$ and showing that $\gamma^{R-1}\bP^{\pi^*}\big[\bQ^* - \bQ^{0}\big]\le  (u/4) \one$.

\subsection{From Halving the Error to Arbitrary Precision}
\label{sub:halving}

In the previous section, we provided an algorithm that on an input policy, outputs a policy with value vector that has $\ell_\infty$ distance to the optimal value vector only half of that of the input one.
%\sidford{Shouldn't we say policy in the preceding sentence.} \lin{How about right now?}
In this section, we give a complete policy computation algorithm by by showing that it is possible to apply this error ``halving'' procedure iteratively. We summarize our meta algorithm in Algorithm~\ref{alg-meta}.
Note that in the algorithm, each call of $\halfErr$ draws new samples from the sampling oracle. 
We refer in this section to Algorithm~\ref{alg-halfErr} as a subroutine \halfErr, which given an input MDP $\cM$ with a sampling oracle, an input value function $\bv^{(i)}$, and an input policy $\pi^{(i)}$, outputs an value function  $\bv^{(i+1)}$ and a policy $\pi^{(i+1)}$.

{\small
\begin{algorithm}[tb!]\caption{Meta Algorithm\label{alg-meta}}
	\begin{algorithmic}[1]
		\State 
		\textbf{Input:} A sampling oracle of some $\cM=(\cS, \cA, \br, \bP, \gamma)$, $\epsilon>0, \delta\in(0,1)$
		\State 
		\textbf{Initialize:} $\bv^{(0)}\gets{\bf 0}$, $\pi^{(0)}\gets$ arbitrary policy, $R\gets \Theta[\log(\epsilon^{-1}(1-\gamma)^{-1})]$
		\For{$i=\{1,2,\ldots,R\}$} 
		\State //\emph{\halfErr~ is initialized with QVI$(u=2^{-i+1}(1-\gamma)^{-1}, \delta, \bv^{(0)}=\bv^{(i-1)}, \pi^{(0)}=\pi^{(i-1)})$}
		\State $\bv^{(i)},\pi^{(i)}\gets\halfErr\gets \bv^{(i-1)},\pi^{(i-1)}$
		\EndFor
		\State 
		\textbf{Output:} $\bv^{(R)},\pi^{(R)}$.
	\end{algorithmic}
\end{algorithm}
}
%\sidford{Should this section and the next be merged?}\lin{Sure}
%\subsection{The Full Algorithm of Obtaining Policy}
%\label{sub:final}
Combining Algorithm~\ref{alg-meta} and Algorithm~\ref{alg-halfErr}, we are ready to present main result.
\begin{theorem}
	\label{thm:dmdp1}
    Let $\cM=(\cS, \cA, \bP, \br, \gamma)$ be a DMDP with a generative model. 
    Suppose we can sample a state from each probability vector $\bP_{s,a}$ within time $O(1)$.	Then for any $\epsilon,\delta\in(0,1)$, there exists an algorithm that halts in time %queries the generative model for the following number of times 
    %{\small
	\[
		T:=O\left[\frac{|\cS||\cA|}{(1-\gamma)^3 \epsilon^2} \log \left(\frac{|\cS||\cA|}{ (1-\gamma)\delta\epsilon}\right)\log\left(\frac{1}{(1-\gamma)\epsilon}\right)\right]
	\]%}
	and obtains a policy $\pi$ such that $\bv^* - \epsilon\one\le \bv^{\pi} \le \bv^*$, with probability at least $1-\delta$ where $\bv^{*}$ is the optimal value of $\cM$. 
	The algorithm uses space $O(|\cS||\cA|)$ and queries the generative model for at most $O(T)$ fresh samples.
	%\sidford{Since is the main theorem, should we have it say the runtime and space explicitly in the theorem?}
\end{theorem}
\begin{remark}
In the above theorem, we require $\epsilon\in(0,1)$. 
For $\epsilon\ge 1$, our sample complexity may fail to be optimal. 
We leave this for a future project.
\end{remark}
\begin{remark}
	The full analysis of the halving algorithm is presented in Section~\ref{sec:half}.
	Our algorithm can be implemented in space $O(|\cS||\cA|)$ since in Algorithm~\ref{alg-halfErr}, the initialization phase can be done for each $(s,a)$ and compute $\bw(s,a), \wt{\bw}(s,a),$ $\wh{\bsigma}(s,a), \bQ^{(0)}(s,a)$ without storing the samples.
	The updates can be computed in space $O(|\cS||\cA|)$ as well.
\end{remark}

%\lin{Please do not read the following section, it is not ready yet.}
%\begin{proof}
%Let $\bv^{(0)} = {\bf0}$ and $\pi^{(0)}$ be an arbitrary policy. 
%We observe that $\norm{\bv^{(0)}-\bv^*}_{\infty}\le (1-\gamma)^{-1}$ and $\bv^{(0)}\le \cT_{\pi}\bv^{(0)}$.
%Denote Algorithm 1 in the last section as $\cA(\bv, \pi, \epsilon, \delta)$.
%Setting $R_1=\lceil c_1\log{(1-\gamma)^{-1}\epsilon^{-1}}\rceil$ for some large constant $c_1$.
%For $1\le i\le R_1$, we run the following procedure,
%\[
%(\bv^{(i)}, \pi^{(i)}) \gets \cA\bigg[\bv^{(i-1)}, \pi^{(i-1)}, \frac{1}{(1-\gamma)\cdot 2^{i-1}}, \frac{\delta}{2R_1}\bigg].
%\]
%By Proposition~\ref{thm:alg 1} and standard conditional probability analysis, we obtain,
%with probability at least $1-\delta$, $\bv^{R_1}\le \cT_{\pi^{(R_1)}}\bv^{R_1}$ and,
%\[
%    \norm{\bv^{(R_1)}-\bv^*}_\infty \le  \frac{1}{(1-\gamma)\cdot 2^{R_1-1}}\le\epsilon
%\]
%provided that $2^{1-R_1}\le (1-\gamma)\cdot \epsilon$.
%Since $\bv^{(R_1)}\le \cT_{\pi^{(R_1)}}\bv^{(R_1)}$, by monotonicity of $\cT_{\pi^{(R_1)}}$, we have
%\[
%\bv^{(R_1)}\le \cT_{\pi^{(R_1)}}\bv^{(R_1)}\le \cT^2_{\pi^{(R_1)}}\bv^{(R_1)} \ldots\le \bv^{\pi^{(R_1)}} \le \bv^*.
%\]
%Hence
%\[
% \bv^* - \bv^{\pi^{(R_1)}} \le  \bv^*- \bv^{(R_1)} \le \epsilon\one.
%\]
%Letting $\pi^{(R_1)}$ be the output, we conclude the proof of the theorem.
%\end{proof}
%
%
%
%

\vspace{-1mm}
\section{Concluding Remark}
In summary, for a discounted Markov Decision Process (DMDP) $\cM=(\cS, \cA, \bP, \br, \gamma)$ provided we can only access the transition function of the DMDP through a generative sampling model, we provide an algorithm which computes an $\epsilon$-approximate optimal (for $\epsilon\in(0,1)$) policy with probability $1 - \delta$ where both the time spent and number of sample taken is upper bounded by $\wt{O}((1-\gamma)^{-3}\epsilon^{-2}|\cS||\cA|)$.
This improves upon the previous best known bounds by a factor of $1/(1 - \gamma)$ and matches the the lower bounds proved in \cite{azar2013minimax} up to logarithmic factors.

The appendix is structured as follows. 
Section~\ref{sec:full model compare} surveys the existing runtime results for solving the DMDP when a full model is given.
 Section~\ref{sec:alg_value} provides an runtime optimal algorithm for computing approximate value functions (by directly combining \cite{azar2013minimax} and \cite{sidford2018variance}).
Section~\ref{sec:var_bounds} gives technical analysis and variance upper bounds for the total-variance technique.
Section~\ref{sec:lower bound} discusses sample complexity lower bounds for obtaining approximate policies with a generative sampling model.
Section~\ref{sec:missing proof} provides proofs to lemmas, propositions and theorems in the main text of the paper. 
Section~\ref{sec:finite_horizon} extends our method and results to the finite-horizon MDP and provides a nearly matching sample complexity lower bound.

\clearpage
\bibliographystyle{alpha}
\bibliography{colt_mdp_ref}

% Acknowledgments---Will not appear in anonymized version
%\acks{We thank a bunch of people.}

\clearpage
\appendix
\section{Previous Work on Solving DMDP with a Full Model}
\label{sec:full model compare}
Value iteration was proposed by \cite{bellman1957dynamic} to compute an exact optimal policy of a given DMDP in time $\mathcal{O}(({1-\gamma})^{-1}|\cS|^2|\cA| L {\log((1-\gamma)^{-1})})$, where $L$ is the total number of bits needed to represent the input; and it can find an approximate $\epsilon$-approximate solution in time $\mathcal{O}(|\cS|^2 |\cA| (1-\gamma)^{-1} \log(1/\epsilon(1-\gamma)))$; see e.g. \cite{tseng1990solving, littman1995complexity}. 
%\sidford{I would be tempted to just give the fastest running times for an explicit DMDP with $\log(1/\epsilon)$ or L dependence and then focus more on sublinear sample bounds.} 
%\sidford{do we have a choice in the citation format? If so, I would opt for the one with all the authors since many of these papers are alphabetical.}
The policy iteration was introduced by  \cite{howard1960dynamic} shortly after, where the policy is monotonically improved according to its associated value function. Its complexity has also been analyzed extensively; see e.g. \cite{mansour1999complexity,ye2011simplex,scherrer2013improved}. Ye \cite{ye2011simplex} showed that policy iteration and the simplex method are strongly polynomial for DMDP and terminates in $\mathcal{O}(|\cS|^2|\cA| (1-\gamma)^{-1}\log(|\cS|(1-\gamma)^{-1}))$ number of iterations. Later \cite{hansen13} and \cite{scherrer2013improved} improved the iteration bound to $O(|\cS||\cA| (1-\gamma)^{-1} \log((1-\gamma)^{1}))$ for Howard's policy iteration method. 
A third approach is to formulate the nonlinear Bellman equation into a linear program \cite{d1963probabilistic, de1960problemes}, and solve it using standard linear program solvers, such as the simplex method by Dantzig \cite{dantzig2016linear} and the combinatorial interior-point algorithm by \cite{ye2005new}. \cite{lee2014path, lee2015efficient} showed that one can solve linear programs in $\wt{O}(\sqrt{\hbox{rank}(A)} )$ number of linear system solves, which, applied to DMDP, yields to a running time of $\wt{O}( |\cS|^{2.5} |\cA| L )$ for computing the exact policy and $\wt{O}( |\cS|^{2.5} |\cA| \log(1/\epsilon))$ for computing an $\epsilon$-optimal policy. \cite{sidford2018variance} further improved the complexity of value iteration by using randomization and variance reduction. See Table~\ref{table-exact} for comparable run-time results or computing the optimal policy when the MDP model is fully given. 
%\sidford{if someone could figure out how to vertically center the text in the table, that would be great.}\lin{Done}
	\begin{table}[htb!]
		\begin{center}
			%	{
			\small
			%	\centering
			%	\begin{tabular}{|c|c|c|}
			\centering
			\begin{tabular}{|>{\centering\arraybackslash}m{4cm}|>{\centering\arraybackslash}m{5cm}|>{\centering\arraybackslash}m{3cm}|}
				\hline
				\textbf{Algorithm} & \textbf{Complexity} & \textbf{References} \\ 
				\hline
				Value Iteration (exact) & $\S^2\A L \frac{\log(1/(1-\gamma))}{1-\gamma}$ & \cite{tseng1990solving, littman1995complexity}\\
				\hline
				Value Iteration & $\S^2\A \frac{\log(1/(1-\gamma)\epsilon)}{1-\gamma}$ & \cite{tseng1990solving, littman1995complexity}\\
				\hline
				%Value Iteration & $\S^2\A L \frac{\log(1/(1-\gamma))}{1-\gamma}$ & \cite{tseng1990solving, littman1995complexity}
				%\\  \hline
				Policy Iteration
				(Block Simplex) & $\frac{\S^4\A^2}{1-\gamma} \log(\frac1{1-\gamma})$  & \cite{ye2011simplex},\cite{scherrer2013improved}
				\\  \hline
				%LP Algorithm  &   $\S^3\A^2 L$ & ??
				%\\  \hline
				& &
				\\[-1em]
				Recent Interior Point Methods & $\wt{O}( \S^{2.5}\A L )$ $\wt{O}( |\cS|^{2.5} |\cA| \log(1/\epsilon))$  & \cite{lee2014path} 
				\\[-1em]
				& &
				\\  \hline
				Combinatorial Interior 
				Point Algorithm   & $\S^4\A^4\log\frac{\S}{1-\gamma}$  
				& \cite{ye2005new}
				\\  \hline
				& &
				\\[-1em]
				High Precision Randomized 
				Value Iteration & $ \wt{O} \bigg[\left(\nnz(P)  +\frac{|\cS| |\cA|}{(1 - \gamma)^3} \right) \log\left(\frac{1}{\epsilon\delta}\right) \bigg]$ & \cite{sidford2018variance}\\[-1em]
				& &
				\\
				\hline
			\end{tabular}
			%	}
		\end{center}
		\caption{
			\small
			\textbf{Running Times to Solve DMDPs Given the Full MDP Model}: In this table, $\S$ is the number of states, $\A$ is the number of actions per state, $\gamma \in(0,1)$ is the discount factor, and $L$ is a complexity measure of the linear program formulation that is at most the total bit size to present the DMDP input. Rewards are bounded between 0 and 1.
			\label{table-exact}}
		\label{tab:literature_runtime_exact}
		\vspace{-1mm}
	\end{table}

\section{Sample and Time Efficient Value Computation}
\label{sec:alg_value}

In this section, we describe an algorithm that obtains an $\epsilon$-optimal values in time $\wt{O}(\epsilon^{-2}(1-\gamma)^{-3}|\cS||\cA|)$. Note that the time and number of samples of this algorithm is optimal (up to logarithmic factors) due to the lower bound in \cite{azar2013minimax} which also established this upper bound on the sample complexity (but not time complexity) of the problem. 

%Note that a near-optimal sample complexity (the number of calls to the sampling oracle) was shown in prior work \cite{azar2013minimax}. The emphasis of their work was on sample complexity and the running time they achieve was simply by using standard value iteration. 
%A straightforward analysis suggests their time complexity is  $\wt{O}(|\cS||\cA|(1-\gamma)^{-4})$. 

We achieve this by combining the algorithms in \cite{azar2013minimax} and \cite{sidford2018variance}. First, we use the ideas and analysis of \cite{azar2013minimax} to construct a sparse MDP where the optimal value function of this MDP approximates the optimal value function of the original MDP and then we run the high precision algorithm in \cite{sidford2018variance} on this sparsified MDP. We show that \cite{sidford2018variance} runs in nearly linear time on sparsified MDP. Since the number of samples taken to construct the sparsified MDP was the the optimal number of samples, to solve the problem, the ultimate running time we thereby achieve is nearly optimal as any algorithm needs spend time at least the number of samples to obtain these samples.

We include this for completeness but note that the approximate value function we show how to compute here does not suffice to compute policy of the MDP of comparable quality.
The greedy policy of an $\epsilon$-optimal value function is an $\epsilon/(1-\gamma)$-optimal policy in the worst case.
It has been shown in \cite{azar2013minimax} that the greedy policy of their value function is $\epsilon$-optimal if $\epsilon \le {(1-\gamma)^{1/2}|\cS|^{-1/2}}$.
However, when $\epsilon$ is so small, the seemingly sublinear runtime $\wt{O}((1-\gamma)^{-3}\cS||\cA|/\epsilon^2)$ essentially means a linear running time and sample complexity as $O((1-\gamma)^{-3}|\cS|^2|\cA|)$. The running time can be obtained by merely applying the result in \cite{sidford2018variance} (although with a slightly different computation model).

\subsection{The Sparsified DMDP}
\label{sec:sparsify mdp}
Suppose we are given a DMDP $\cM=(\cS, \cA, \br, \bP, \gamma)$ with a sampling oracle. To approximate the optimal value of this MDP, we perform a spasification procedure as in \cite{azar2013minimax}. 
Sparsification of DMDP is conducted as follows. Let $\delta>0,\epsilon>0$ be arbitrary. First we pick a number 
\begin{align}
\label{eqn:num samples per sa}
m=\Theta\left[\frac{1}{(1-\gamma)^3 \epsilon^2} \log \left(\frac{|\cS||\cA|}{\delta}\right)\right] ~.
\end{align}
For each $s\in \cS$ and each $a\in \cA$, we generate a sequence of independent samples from $\cS$ using the probability vector $\bP_{s,a}$
\[
s_{s,a}^{(1)}, s_{s,a}^{(2)}, \ldots, s_{s,a}^{(m)}.
\]
Next we construct a new and sparse probability vector $\wh{\bP}_{s,a}\in\Delta_{|\cS|}$ as
\[
\forall s'\in \cS: \wh{\bP}_{s,a}(s') = \frac{1}{m}\cdot\sum_{i=1}^m \one(s_{s,a}^{(i)}=s').
\]
Combining these $|\cS||\cA|$ new probability vectors, we obtain a new probability transition matrix $\wh{\bP}\in\RR^{\cS\times\cA\times\cS}$ with number of non-zeros 
\[
\nnz(\wh{\bP}) =
O\left[\frac{|\cS||\cA|}{(1-\gamma)^3 \epsilon^2} \log \left(\frac{|\cS||\cA|}{\delta}\right)\right] ~.
\]
Denote $\wh{\cM} = (\cS, \cA, \br, \wh{\bP}, \gamma)$ as the sparsified DMDP.
In the rest of this section, we use $\wh{\cdot}$ to represent the quantities corresponding to DMDP $\wh{\cM}$, e.g., $\wh{\bv}^*$ for the optimal value function, $\wh{\pi}^*$ for a optimal policy, and $\wh{\bQ}^*$ for the optimal $Q$-function. %\sidford{Where is $Q$-function defined?}
There is a strong approximation guarantee of the optimal $Q$-function of the sparsified MDP, presented as follows.

\begin{theorem}[\cite{azar2013minimax}]
	\label{thm:qvalue approx}
	Let $\cM$ be the original DMDP and $\wh{\cM}$ be the corresponding sparsified version.
	Let $\bQ^*$ be the optimal $Q$-function vector of the original DMDP and $\wh{\bQ}^*$ be the optimal $Q$-function of $\wh{\cM}$.
	Then with probability at least $1-\delta$ (over the randomness of the samples), 
	\[
	\|{\wh{\bQ}^* - \bQ^*}\|_{\infty}\le \epsilon.
	\]
\end{theorem}
%Next, we show the following theorem, \sidford{I would maybe find a way to restate there theorem and then just use it (if possible) instead of something in between.}
%The proof of this theorem is essentially given in  \cite{azar2013minimax}. 
%Theorem~1 of their paper shows that with probability at least $1-\delta$,
%\[
%\|{\bQ^*-\wh{\bQ}^*}\|_{\infty}\le \epsilon.
%\]
%Thus conditioning on this event,
Recall that ${\bv}^*$ and $\wh{\bv}^*$ are the optimal value functions of $\cM$ and $\wh{\cM}$.
From Theorem~\ref{thm:qvalue approx}, we immediately have
\[
\forall s\in \cS:~ |\bv^*(s)-\wh{\bv}^*(s)|= | \max_{a\in\cA} \bQ^*(s, a) - \max_{a\in\cA} \wh{\bQ}^*(s, a)|
\le \max_{a\in \cA}| \bQ^*(s, a) - \wh{\bQ}^*(s, a)| \le \epsilon,
\]
with probability at least $1-\delta$.

\subsection{High Precision Algorithm in the Sparsified MDP}

Next we shall use the high precision algorithm of the \cite{sidford2018variance} which has the following guarantee.

\begin{theorem}[\cite{sidford2018variance}]
	\label{thm:high precision}
	There is an algorithm which given an input DMDP $\cM = (\cS, \cA, \br, \bP, \gamma)$ in time\footnote{$\wt{O}(f)$ denotes $O(f\cdot \log^{O(1)} f)$.} \[
	\wt{O}\bigg[\bigg(\nnz(\bP) + \frac{|\cS||\cA|}{(1-\gamma)^3}\bigg)\cdot \log\epsilon^{-1}\cdot\log\delta^{-1}\bigg]
	\] %t\sidford{There are more factors in our running time right, there is the low order part as well. I would write this as well as the dependence on $\epsilon$ out explicitly.}
 and outputs a vector $\wt{\bv}^*$ 
	such that with probability at least $1-\delta$, 
	\[
	\|{\wt{\bv}^* - \bv^*}\|_{\infty}\le \epsilon.
	\]
where $\bv^*$ is the optimal value of $\cM$. %\sidford{This line is redundant due to preliminaries right?}
\end{theorem}
Combining the above two theorems, we immediately obtain an algorithm for finding $\epsilon$-optimal value functions. It works by first generating enough samples for each state-action pair and then call the high-precision MDP solver by \cite{sidford2018variance}. It does not sample transitions adaptively. We show that it achieves an optimal running time guarantee (up to $\poly\log$ factors) of obtaining the value function under the sampling oracle model. 

\begin{theorem}
	\label{thm:value}
	Given an input DMDP $\cM=(\cS, \cA, \br, \bP, \gamma)$ with a sampling oracle and optimal value function $\bv^*$, there exists an algorithm, that runs in time
	\[
		\wt{O}\bigg(\frac{|\cS||\cA|}{(1-\gamma)^3}\cdot\frac{1}{\epsilon^2}\cdot \log^2 \left(\frac{1}{\delta}\right)\bigg)
	\]
	and outputs a vector $\wh{\bv}^*$ 
	such that $\|{\wh{\bv}^* - \bv^*}\|_{\infty}\le O(\epsilon)$ with probability at least $1-O(\delta)$.
\end{theorem}
\begin{proof}
We first obtain a sparsified MDP $\wh{\cM}=(\cS, \cA, \br, \wh{\bP}, \gamma)$ using the procedure described in Section~\ref{sec:sparsify mdp}.
This procedure runs in time $O(|\cS||\cA|m)$, recalling that $m$ is the number of samples per $(s,a)$, defined in \eqref{eqn:num samples per sa}.
Let $\wh{\bu}^*$ be the optimal value function of $\wh{\cM}$.
By Theorem~\ref{thm:qvalue approx}, with probability at least $1-\delta$, $\|\wh{\bu}^* -\bv^*\|\le \epsilon$, which we condition on for the rest of the proof.
Calling the algorithm in Theorem~\ref{thm:high precision}, we obtain a vector $\wt{\bu}^*$ in time 
\[
\wt{O}\bigg[\bigg(\nnz(\wh{\bP}) + \frac{|\cS||\cA|}{(1-\gamma)^3}\bigg)\cdot \log\epsilon^{-1}\cdot\log\delta^{-1}\bigg] 
=  \wt{O}\bigg(\frac{|\cS||\cA|}{(1-\gamma)^3}\cdot\frac{1}{\epsilon^2}\cdot \log^2\frac{1}{\delta}\bigg)
\] 
and that with probability at least $1-\delta$,
$\|\wt{\bu}^* - \wh{\bu}^*\|\le \epsilon$, which we condition on.
By triangle inequality, we have
\[
\|\wt{\bu}^* - \bv^*\|_{\infty}\le \|\wt{\bu}^* - \wh{\bu}^*\|_{\infty} + \|\wh{\bu}^* - {\bv}^*\|_{\infty}\le 2\epsilon.
\]
This concludes the proof.
\end{proof}
\section{Variance Bounds}
\label{sec:var_bounds}

%\sidford{I think I would rearrange this section explaining first what exactly the main theorem of this section is, i.e. bounding the variance term. I would then give the technical lemma that reduces it to bounding the variance of playing a policy, and then introduce and bound that term. Right now I think we essentially do things in the reverse order.}\lin{I have rearranged this section.}

%\subsection{Transition Kernels and $Q$-functions}
In this section, we study some properties of a DMDP. 
Most of the content in this section is similar to \cite{azar2013minimax}.
We provide slight modifications and improvement to make the results fit to our application.
The main result of this section is to show the following lemma.
\begin{lemma}[Upper Bound on Variance]
	\label{lemma:variance bound}
	For any $\pi$, we have
	\[
	\big\|(\bI - \gamma\bP^{\pi})^{-1}\sqrt{{\bsigma}_{\bv^{\pi}}}\big\|_{\infty}^2 \le \frac{1 + \gamma}{\gamma^2(1 - \gamma)^3},
	\]
	where $\sigma_{v^{\pi}}= \bP^{\pi} (\bv^\pi)^2 - (\bP^{\pi}\bv^{\pi})^2$ is the ``one-step'' variance of playing policy $\pi$.
\end{lemma}
Before we prove this lemma, we introduce another notation.
%For any vector $\bv\in\RR^{\cS}$, we slightly abuse the notation by defining vector ${\bsigma}_{\bv}\in \RR^{|\cS||\cA|}$ as 
%\sidford{again should the variance just be defined without the $\gamma^2$ to make this a little cleaner?}\lin{in our main text, the $\gamma^2$ is removed. Here we need this to have a nice form like Lemma~\ref{lem:var_bell}} \sidford{Why won't it still be nice form with just the $\gamma$ factors rearanged if other definition?}\lin{For now, it looks like a Bellman equation for MDP, with a reward vector $\wt{\bsigma}$, discount vector $\gamma^2$, right?}
%\[
%\forall (s, a)\in \cS\times\cA:\quad {\bsigma}_{\bv}(s, a) =\underset{s'\sim \bP_{s, a}}{\var}[\bv(s')] =  \sum_{s'}\bP_{s, a}(s')\left(\bv(s') - \bP_{s, a}^\top \bv\right)^2.
%\]
%Note any policy $\pi$, we have
%\[
%\wt{\bsigma}_{\bv^\pi} =  \gamma^2\bP (\bv^\pi)^2 - \gamma^2(\bP\bv^{\pi})^2 =  \gamma^2\bP^{\pi}(\bQ^{\pi})^2 - \gamma^2 (\bP^{\pi}\bQ^{\pi})^2
%\]
We define $\bSigma^{\pi}\in \RR^{|\cS||\cA|}$ for all $ (s, a) \in \cS\times\cA$ by
\[
\bSigma^{\pi}(s, a) := \EE\bigg[{\bigg(\br(s, a) + \sum_{t\ge 1}\gamma^{t}\br(s^t, a^t) - \bQ^{\pi}(s, a)\bigg)^2\bigg| s^0 =s, a^0 = a, a^t=\pi(s^t)}\bigg]
\]
where $a^t = \pi(s^t)$.
Thus $\bSigma^{\pi}$ is the variance of the reward of starting with $(s, a)$ and play $\pi$ for infinite steps. 
The crucial observation of obtaining the near-optimal sample complexity is the following ``Bellman Equation'' for variance. 
It is a consequence of ``the law of total variance''.
\begin{lemma}[Bellman Equation for variance]
	\label{lem:var_bell}
	$\bSigma^{\pi}$ satisfies the Bellman equation
	\[
	\bSigma^{\pi} = \gamma^2{\bsigma}_{\bv^\pi} + \gamma^2\cdot \bP^{\pi} \bSigma^\pi.
	\]
\end{lemma}
\begin{proof}
	By direct expansion,
	\begin{align}
	\bSigma^{\pi}(s,a) &= \EE\bigg[{\big({\br(s, a) + \sum_{t\ge 1}\gamma^{t}\br(s^t, a^t)}\big)^2 \bigg| s^0 =s, a^0 = a, a^t=\pi(s^t)}\bigg]
	- (\bQ^{\pi}(s,a))^2.
	\end{align}
	The first term in RHS can be written as
	\begin{align*}
	\EE&\bigg[{\bigg({\br(s, a) + \sum_{t\ge 1}\gamma^{t}\br(s^t, a^t)}\bigg)^2\bigg| s^0 =s, a^0 = a, a^t=\pi(s^t)} \bigg]\\
	&=\sum_{s'\in \cS}\bP_{s,a}(s')\EE
	\bigg[{\bigg({\br(s, a) + \gamma\br(s',\pi(s')) + \gamma\sum_{t\ge 1}\gamma^{t}\br(s^t, a^t)}\bigg)^2\bigg| s^0 =s', a^0 = \pi(s'), a^t=\pi(s^t)}\bigg]\\
	&=\br(s,a)^2 + 2\gamma \br(s,a)\cdot \sum_{s'\in \cS}\bP_{s,a}(s')\bQ^{\pi}(s', \pi(s')) \\
	& \qquad
	+ \gamma^2\sum_{s'\in \cS}\bP_{s,a}(s')\EE
	\bigg[{\bigg({\br(s',\pi(s'))  + \sum_{t\ge 1}\gamma^{t}\br(s^t, a^t)}\bigg)^2\bigg| s^0 =s', a^0 = \pi(s'), a^t=\pi(s^t)}\bigg]\\
	&=\br(s,a)^2 + 2\gamma \br(s,a)\cdot \sum_{s'\in \cS}\bP_{s,a}(s')\bQ^{\pi}(s', \pi(s')) 
	+ \gamma^2(\bP^{\pi}\bSigma^{\pi})(s,a) + \gamma^2 \sum_{s'\in \cS}\bP_{s,a}(s')(\bQ^{\pi}(s', \pi(s')) )^2\\
	&=\bQ^{\pi}(s, a)^2 + \gamma^2(\bP^{\pi}\bSigma^{\pi})(s,a)
	+ \gamma^2 \sum_{s'\in \cS}\bP_{s,a}(s')(\bQ^{\pi}(s', \pi(s')) )^2 - \gamma^2 \bigg({\sum_{s'\in \cS}\bP_{s,a}(s')\bQ^{\pi}(s', \pi(s'))}\bigg)^2\\
	& = \bQ^{\pi}(s, a)^2 + \gamma^2(\bP^{\pi}\bSigma^{\pi})(s,a) + \gamma^2{\bsigma}_{\bv^{\pi}}(s,a).
	\end{align*}
	Combining the above two equations, we conclude the proof.
\end{proof}
 
 As a remark, we note that
 \[
 \bSigma^{\pi} = \gamma^2(\bI - \gamma^2\bP^{\pi})^{-1}{\bsigma}_{\bv^{\pi}}.
 \]
 Furthermore, by definition, we have
 \[
  \max_{(s, a)\in \cS}\bSigma^{\pi}(s,a) \le (1-\gamma)^{-2},
 \]
 %where $\beta = (1-\gamma)^{-1}$.
 The next lemma is crucial in proving the error bounds. 

\begin{lemma}
	\label{lemma:sqrv}
	Let $\bP \in \R^{n \times n}$ be a non-negative matrix in which every row has $\ell_1$ norm at most $1$, i.e. $\ell_\infty$ operator norm at most $1$. Then for all  $\gamma \in(0, 1)$ and $\bv \in \R^{n}_{\geq 0}$ we have  
	\[
	\| (\bI - \gamma \bP)^{-1} \sqrt{\bv} \|_\infty 
	\leq  \sqrt{\frac{1}{1 - \gamma} \left\| (\bI - \gamma \bP )^{-1} \bv \right\|_\infty}
	\leq  \sqrt{\frac{1 + \gamma}{1 - \gamma} \left\| (\bI - \gamma^2 \bP )^{-1} \bv \right\|_\infty} ~.
	\] 
\end{lemma}

\begin{proof}
	Since, every row of $\bP$ has $\ell_1$ norm at most $1$, by Cauchy-Schwarz for $i \in [n]$ we have 
	\[
	[\bP \sqrt{\bv}]_i = \sum_{j \in [n]} \bP_{ij} \sqrt{\bv}_j
	\leq 
	\sqrt{\sum_{j \in [n]} \bP_{ij} \cdot \sum_{j \in [n]} \bP_{ij} \bv_j }
	\leq \sqrt{\bP \bv} ~.
	\]
	Since $\bv$ is non-negative and applying $\bP$ preserves non-negativity, applying this inequality repeatedly yields that $\bP^{k} \sqrt{\bv} \leq \sqrt{\bP^{k} \bv}$ entrywise for all $k >0$. Consequently, Cauchy-Schwarz again yields
	\begin{align*}
	(\bI - \gamma \bP)^{-1} \sqrt{\bv}
	&= 
	\sum_{i = 0}^{\infty} \left[\gamma \bP\right]^{i} \sqrt{\bv}
	\leq \sum_{i = 0}^{\infty} \gamma^{i} \sqrt{\bP^{i} \bv}
	\leq 
	\sqrt{
		\sum_{i = 0}^{\infty}
		\gamma^i 
		\cdot 
		\sum_{i = 0}^{\infty} \gamma^i \bP^i v
	}\\
	&\le \sqrt{\frac{1}{1 - \gamma} \left\| (\bI - \gamma \bP )^{-1} \bv \right\|_\infty}
	~.
	\end{align*}
	Next, as $(\bI - \gamma \bP)(\bI + \gamma \bP) = (\bI - \gamma \bP^2)$ we see that $(\bI - \gamma \bP)^{-1} = (\bI + \gamma \bP) (\bI - \gamma^2 \bP)^{-1}$. Furthermore, as $\|\bP x\|_\infty \leq \|\bx\|_\infty$ for all $\bx$ we have $\|(\bI + \gamma \bP) x\|_\infty \leq (1 + \gamma) \| x\|_\infty$ for all $\bx$ and therefore $\| (\bI - \gamma \bP )^{-1} \bv \|_\infty \leq (1 + \gamma) \| (\bI - \gamma^2 \bP )^{-1} \bv \|_\infty$ as desired. 
\end{proof}
We are now ready to prove Lemma~\ref{lemma:variance bound}.
\begin{proof}[Proof of Lemma~\ref{lemma:variance bound}]
	The lemma follows directly from the application of Lemma~\ref{lemma:sqrv}. This proof is slightly simpler, tighter, and more general than the one in \cite{azar2013minimax}.
\end{proof}

\section{Lower Bounds on Policy}
\label{sec:lower bound}

\begin{lemma}
Suppose $\cM=(\cS, \cA, P, \gamma, \br)$ is a DMDP with an sampling oracle. Suppose $\pi$ is a given policy.
Then there is an algorithm, halts in $\wt{O}((1-\gamma)^{-3}\epsilon^{-2}|\cS|)$ time, outputs a vector $\bv$ such that, with high probability,
$\|\bv^\pi - \bv\|_\infty \le \epsilon$.
\end{lemma}
\begin{proof}
The lemma follows from a direct application of Theorem~\ref{thm:high precision}.
\end{proof}

\begin{remark}
	Suppose $|\cA|= \wt{\Omega}(1)$. 
	Suppose there is an algorithm that obtains an $\epsilon$-optimal policy with $Z$ samples, then the above lemma implies an algorithm for obtaining an $\epsilon$-optimal value function with $Z + \wt{O}((1-\gamma)^{-3}\epsilon^{-2}|\cS|)$ samples.
	By the $\Omega((1-\gamma)^{-3}\epsilon^{-2}|\cS||\cA|)$ sample bound on obtaining approximate value functions given in \cite{azar2013minimax}, the above lemma implies a 
	\[
	Z = \Omega((1-\gamma)^{-3}\epsilon^{-2}|\cS||\cA|) - \wt{O}((1-\gamma)^{-3}\epsilon^{-2}|\cS|) = \Omega((1-\gamma)^{-3}\epsilon^{-2}|\cS||\cA|)
	\] sample lower bound for obtaining an $\epsilon$-optimal policy. 
	
	%\sidford{I would explain just a little more, i.e. fully explain the contradiction.}\lin{done}
\end{remark}

\section{Missing Proofs}
\label{sec:missing proof}
Here are several standard properties of the Bellman value operator (see, e.g., \cite{bertsekas2013abstract}). %\sidford{citation?}
\begin{fact}
	Let $\bv_1, \bv_2\in \RR^{\cS}$ be two vectors. Let $\cT$ be a value operator of a DMDP with discount factor $\gamma$. 
	Let $\pi\in \cA^{\cS}$ be an arbitrary policy.
	Then the follows hold.
	\begin{itemize}
		\item\textbf{Monotonicity}: 	If $\bv_1 \le \bv_2$ then $\cT(\bv_1) \le \cT(\bv_2)$;
		\item\textbf{Contraction}:	 $\|\cT(\bv_1) - \cT(\bv_2)\|_{\infty}\le \gamma \|\bv_1 - \bv_2\|_{\infty}$ and $\|\cT_{\pi}(\bv_1) - \cT_{\pi}(\bv_2)\|_{\infty}\le \gamma \|\bv_1 - \bv_2\|_{\infty}$.
		%\item\textbf{Contraction}: 	 $\|\cT_{\pi}(\bv_1) - \cT_{\pi}(\bv_2)\|_{\infty}\le \gamma \|\bv_1 - \bv_2\|_{\infty}$. \sidford{Isn't this just a special case of the previous bullet?}
	\end{itemize}
\end{fact}

\subsection{Missing Proofs from Section~\ref{sec:alg_policy}}
\label{sec:proof of main alg}
To begin, we introduce two standard concentration results.
Let $\bp\in \Delta_{\cS}$ be a probability vector, and $\bv\in\RR^{\cS}$ be a vector. 
Let ${\bp}_m\in\Delta_{\cS}$ be empirical estimations of $\bp$ using $m$ i.i.d. samples from the distribution $\bp$.
For instance, let these samples be $s_1, s_2, \ldots, s_m\in \cS$, then $\forall s\in \cS: {\bp}_m(s) = \sum_{j=1}^m\one(s_j = s)/m$.

\begin{theorem}[Hoeffding Inequality]
	\label{them:hoeffding}
	Let $\delta\in (0,1)$ be a parameter, vectors $\bp, \bp_m$ and $\bv$ defined above. 
	Then with probability at least $1-\delta$,
	\[
	\big|{\bp^{\top}\bv - \bp_m^{\top}\bv}\big| \le \inorm{\bv}\cdot\sqrt{{2m^{-1}\log(2\delta^{-1})}}.
	\]
\end{theorem}
\begin{theorem}[Bernstein Inequality]
	\label{them:bernstein}
	Let $\delta\in (0,1)$ be a parameter, vectors $\bp, \bp_m$ and $\bv$ defined as in Theorem~\ref{them:hoeffding}. 
	Then with probability at least $1-\delta$ %\sidford{We should be consistent with variance notation, I would rather use Var like above then the V we used iearlier in the paper.}
	\[
	\big|{\bp^{\top}\bv - \bp_m^{\top}\bv}\big| \le \sqrt{2m^{-1}\underset{s'\sim \bp}{\var}(\bv(s'))\cdot \log({2}{\delta^{-1}})} + ({2}/{3})m^{-1}{\inorm{\bv}\cdot\log(2\delta^{-1})},
	\]
	where $\underset{s'\sim \bp}{\var}(\bv(s')) = \bp^\top \bv^2 - (\bp^\top \bv)^2$.
	%where $\sigma_{\bv} = \gamma^2\underset{s'\sim \bp}{\var}(\bv(s'))$ is the scaled variance of each sample.
\end{theorem}

\begin{proof}[Proof of Lemma~\ref{lemma: emprical mean}]
	By Theorem~\ref{them:bernstein} and a union bound over all $(s,a)$ pairs, with probability at least $1-\delta/4$, for every $(s,a)$, we have
	\begin{align}
	\label{eq:proof emprical 1}
	\big|{\wt{\bw}(s,a) - \bP_{s,a}^\top\bv^{(0)}}\big|\le\sqrt{{2\sigma_{\bv^{(0)}}\cdot m_1^{-1}\cdot{L}}} +
	{2}\cdot (3m_1)^{-1}\cdot\norm{\bv^{(0)}}_{\infty}\cdot {L},
	\end{align}
	which is the first inequality.
	
	Next, by Theorem~\ref{them:hoeffding} and a union bound over all $(s,a)$ pairs, with probability at least $1-\delta/4$, for every $(s,a)$, we have
	\[
	\big|{\wt{\bw}(s,a) - \bP_{s,a}^\top\bv^{(0)}}\big|\le \norm{\bv^{(0)}}_{\infty} \cdot \sqrt{{2m_1^{-1}{L}}},
	\]
	which we condition on.
	Thus
	\begin{align*}
	\big|\wt{\bw}(s,a)^2 - (\bP_{s,a}^\top\bv^{(0)})^2\big| &= 
	(\wt{\bw}(s,a) + \bP_{s,a}^\top\bv^{(0)})\cdot |\wt{\bw}(s,a) - \bP_{s,a}^\top\bv^{(0)}|\\
	&\le \bigg[2 \bP_{s,a}^\top\bv^{(0)} + \norm{\bv^{(0)}}_{\infty} \cdot \sqrt{{2m_1^{-1}{L}}}\bigg]\cdot|\wt{\bw}(s,a) - \bP_{s,a}^\top\bv^{(0)}|\\
	&\le
	2( \bP_{s,a}^\top\bv^{(0)})\cdot\norm{\bv^{(0)}}_{\infty} \cdot \sqrt{{2m_1^{-1}{L}}} + \norm{\bv^{(0)}}^2_{\infty} \cdot {{2m_1^{-1}{L}}}.
	\end{align*}
	Since $\bP_{s,a}^\top\bv^{(0)}\le \norm{\bv^{(0)}}_\infty$, we obtain
	\[
	\big|\wt{\bw}(s,a)^2 - (\bP_{s,a}^\top\bv^{(0)})^2\big| \le 3\norm{\bv^{(0)}}_{\infty}^2 \cdot \sqrt{{2m_1^{-1}{L}}},
	\]
	provided $2m_1^{-1}{L}\le 1$.
	Next by Lemma~\ref{them:hoeffding} and a union bound over all $(s,a)$ pairs, with probability at least $1-\delta/4$, for every $(s,a)$, we have 
	\[
	\left|\frac{1}{m_1} \sum_{j=1}^{m_1}\bv^2(s_{s,a}^{(j)}) - \bP_{s,a}^\top\bv^2\right|\le \norm{\bv^{(0)}}_{\infty}^2 \cdot \sqrt{{2L / m_1}}.
	\]
	By a union bound, we obtain, with probability at least $1-\delta/2$,
	\begin{align}
	\label{eq:proof emprical 2}
	\big|\wh{\bsigma}(s,a) - \bsigma_{\bv^{(0)}}(s, a)\big| 
	&\le \big|\wt{\bw}(s,a)^2 - (\bP_{s,a}^\top\bv^{(0)})^2\big| + \big|{m_1^{-1}}\sum_{j=1}^{m_1}\bv^2(s_{s,a}^{(j)}) - \bP_{s,a}^\top\bv^2\big|\nonumber\\
	&\le 4\norm{\bv^{(0)}}_{\infty}^2 \cdot \sqrt{{2m_1^{-1}{L}}}.
	\end{align}
	By a union bound, with probability at least $1-\delta$, both \eqref{eq:proof emprical 1} and \eqref{eq:proof emprical 2} hold, concluding the proof.
\end{proof}

\begin{proof}[Proof of Lemma~\ref{lemma: variance triangle}]
	Since for each $(s,a)$, $\bsigma_{\bv}(s,a)$ is a variance, then 
	we have triangle inequality,
	\[
	\sqrt{\bsigma_{\bv}} \le \sqrt{\bsigma_{\bv^*}} + \sqrt{\bsigma_{\bv-\bv^*}}.
	\]
	Observing that
	\[
	\bsigma_{\bv-\bv^*}(s,a) \le \bP_{s,a}^\top(\bv-\bv^*)^2
	\le \epsilon^2\cdot\one.
	\]
	We conclude the proof by taking a square root of all three sides of the above inequality.
\end{proof}

\begin{proof}[Proof of Lemma~\ref{lemma:bounds on g}]
	Recall that for each $(s, a) \in \cS\times \cA$, 
	\[
	\bg^{(i)}(s,a)= \frac{1}{m_2} \sum_{j=1}^{m_2} \big[\bv^{(i)}(s_{s,a}^{(j)}) - \bv^{(0)}(s_{s,a}^{(j)}) \big]- (1-\gamma)\frac{u}{8} ~,
	\] where $m_2 = 128(1-\gamma)^{-2}\cdot\log(2|\cS||\cA|R/\delta)$ and $s_{s,a}^{(1)}, s_{s,a}^{(2)}, \ldots, s_{s,a}^{(m_2)}$ is a sequence of independent samples from $\bP_{s,a}$.
	Thus by Theorem~\ref{them:hoeffding} and a union bound over $\cS\times \cA$, with probability at least $1-\delta/R$, we have
	\begin{align*}
	\forall (s,a)\in\cS\times\cA: \bigg|\sum_{j=1}^{m_2} \big[\bv^{(i)}(s_{s,a}^{(j)}) &- \bv^{(0)}(s_{s,a}^{(j)}) \big] - \bP_{s,a}^{\top} \big[\bv^{(i)}  - \bv^{(0)}\big]\bigg| \\
	&\le \norm{\bv^{(i)} - \bv^{(0)}}_\infty\sqrt{{2m_2^{-1}\log(2|\cS||\cA|\delta'^{-1})}}
	%\\
	%&
	\le (1-\gamma)u/8.
	\end{align*}
	Finally by shifting the estimate to have one-sided error, we obtain the one-side error $(1-\gamma)u/4$ in the statement of this lemma.
	%\sidford{Seems like there is an $8$ here but a $4$ in the statement.}\lin{this is unbiased, need a shift to make it one-sided error.}\sidford{I see, maybe add a little english to that effect?}
\end{proof}

\begin{proof}[Proof of Lemma~\ref{lemma:induction lemma}]
	%We prove by induction on $i$.
	For $i=0$, $\bQ^{(0)} = \br + \gamma \bw$.
	By Lemma~\ref{lemma: emprical mean}, with probability at least $1-\delta$,
	\begin{align*}
	|{\wt{\bw} - \bP\bv^{(0)}}| \le \sqrt{2\alpha_1\bsigma_{\bv^{(0)}}} +
	\frac{2}{3}\cdot\alpha_1\cdot \norm{\bv^{(0)}}_{\infty}\one,
	\end{align*}
	and
	\begin{align}
	\label{eqn:var}
	\big|\wh{\bsigma}- \bsigma_{\bv^{(0)}}\big| \le 4\norm{\bv^{(0)}}_{\infty}^2 \cdot \sqrt{2\alpha_1} \one,
	\end{align}
	which we condition on. 
	We have
	\[
	|{\wt{\bw} - \bP\bv^{(0)}}| \le  \sqrt{2\alpha_1\wh{\bsigma}} + (4\alpha_1^{3/4}\norm{\bv^{(0)}}_\infty +
	\frac{2}{3}\cdot\alpha_1\cdot \norm{\bv^{(0)}}_{\infty})\one.
	\]
	Thus 
	\begin{align}
	\label{eqn:upper w}
	\bw = \wt{\bw} - \sqrt{2\alpha_1\wh{\bsigma}} - 4\alpha_1^{3/4}\norm{\bv^{(0)}}_\infty\one -
	\frac{2}{3}\cdot\alpha_1\cdot \norm{\bv^{(0)}}_{\infty}\one 
	\le  \bP\bv^{(0)},
	\end{align}
	and
	\begin{align*}
	\bw\ge \bP\bv^{(0)} - 2\sqrt{2\alpha_1\wh{\bsigma}} - (8\alpha_1^{3/4}\norm{\bv^{(0)}}_\infty +
	\frac{4}{3}\cdot\alpha_1\cdot \norm{\bv^{(0)}}_{\infty})\one
	. 
	\end{align*}
	By \eqref{eqn:var} and Lemma~\ref{lemma: variance triangle}, we have
	\[
	\sqrt{\wh{\bsigma}} \le \sqrt{\bsigma_{\bv^{(0)}}} + 2 \norm{\bv^{(0)}}_{\infty} (2\alpha)^{1/4}\one
	\le \sqrt{\bsigma_{\bv^{*}}} +  u \one+ 2 \norm{\bv^{(0)}}_{\infty} (2\alpha)^{1/4}\one.
	\]
	we have
	\begin{align}
	\label{eqn:lower w}
	\bw\ge \bP\bv^{(0)} - 2\sqrt{2\alpha_1\bsigma_{\bv^*}} -2\sqrt{2\alpha_1}u\one- 16\alpha_1^{3/4}\norm{\bv^{(0)}}_\infty\one -
	\frac{4}{3}\cdot\alpha_1\cdot \norm{\bv^{(0)}}_{\infty}\one 
	\end{align}
	
	%Next, by definition of $\bv^{(i)}$ (Line~\ref{alg: v1} and \ref{alg: v2}), we have
	%\[
	%\bv^{(0)}\le \bv^{(1)}.
	%\]
	For the rest of the proof, we condition on the event that \eqref{eqn:upper w} and \eqref{eqn:lower w} hold, which happens with probability at least $1-\delta$.
	Denote $\bv^{(-1)}={\bf 0}$.
	Thus we have $\bv^{(-1)}\le \bv^{(0)} \le \cT_{\pi^{(0)}} (\bv^{(0)})$.
	Next we prove the lemma by induction on $i$.
	Assume for some $i\ge 1$, %and assume for each $0\le k\le i-1$, 
	with probability at least $1-(i-1)\delta'$ the following holds, 
	\[
	\forall 0\le k\le i-1:\quad\bv^{(k-1)}\le \bv^{(k)} \le \cT_{\pi^{(k)}} (\bv^{(k)}),
	\]
	which we condition on.
	Next we show that the lemma statement holds for $k=i$.
	By definition of $\bv^{(i)}$ (Line~\ref{alg: v1} and \ref{alg: v2}),
	\[
	\bv^{(i-1)}\le \bv^{(i)} \quad\text{and}\quad  \bv\big(\bQ^{(i-1)}\big)\le \bv^{(i)}.
	\]
	Furthermore, since $\bv^{(0)}\le \bv^{(1)} \le \ldots \le \bv^{(i-1)}\le \cT_{\pi^{i-1}}\bv^{(i-1)} \le \cT\bv^{(i-1)}\le \cT^{\infty}\bv^{(i-1)} = \bv^*$, we have
	\[
	\bv^{(i)}-\bv^{(0)} \le \bv^{*} - \bv^{(0)} \le u\one.
	\]
	By Lemma~\ref{lemma:bounds on g}, we have, with probability at least $1-\delta'$
	\begin{align}
	\label{eqn:bound g}
	\bP  \big[\bv^{(i)}  - \bv^{(0)}\big] - \frac{(1-\gamma) u}{8}\cdot\one\le \bg^{(i)} \le \bP  \big[\bv^{(i)}  - \bv^{(0)}\big],
	\end{align}
	which we condition on for the rest of the proof.
	Thus we have
	\[
	\bQ^{(i)} = \br + \gamma(\bw + \bg^{(i)}) \le \br + \gamma(\bP\bv^{(0)} + \bP\bv^{(i)} - \bP\bv^{(0)}) = \br + \gamma \bP \bv^{(i)}.
	\]
	To show $\bv^{(i)}\le \cT_{\pi^{(i)}} \bv^{(i)}$, we notice that if for some $s$,  $\pi^{(i)}(s)\neq \pi^{(i-1)}(s)$, then
	\[
	\bv^{(i)}(s) \le [\cT_{\pi^{(i)}} \bv^{(i-1)}](s) \le  [\cT_{\pi^{(i)}} \bv^{(i)}](s),
	\]
	where the first inequality follows from $\bv^{(i)}(s) \le \br(s, \pi^{(i)}(s))+\gamma\bP_{s, \pi^{(i)}(s)}^\top\bv^{(i-1)} = \cT_{\pi^{(i)}} \bv^{(i-1)}$.
	On the other hand, if $\pi^{(i)}(s)= \pi^{(i-1)}(s)$, then 
	\[
	\bv^{(i)}(s) = \bv^{(i-1)}(s) \le (\cT_{\pi^{(i-1)}} \bv^{(i-1)})(s) \le (\cT_{\pi^{(i-1)}} \bv^{(i)})(s) = (\cT_{\pi^{(i)}} \bv^{(i)})(s).
	\]
	This completes the induction step.
	Lastly, combining \eqref{eqn:lower w} and \eqref{eqn:bound g}, we have
	\begin{align*}
	\bQ^* - \bQ^{(i)} &= \bQ^{*} - \br - \gamma (\bw + \bg^{(i)}) 
	= \gamma\bP\bv(\bQ^*) - \gamma (\bw + \bg^{(i)}) \\
	&=\gamma\bP\bv(\bQ^*) - \gamma \bP(\bv^{(i)} -\bv^{(0)}) - \gamma \bP \bv^{(0)} + \bxi^{(i)}\\
	&=\gamma\bP\bv(\bQ^*) - \gamma \bP\bv^{(i)} + \bxi^{(i)},
	\end{align*}
	where
	\[
	\bxi^{(i)} \le {(1-\gamma) u}/{8}\cdot \one + 2\sqrt{2\alpha_1\bsigma_{\bv^*}} +2\sqrt{2\alpha_1} u\cdot\one+ 16\alpha_1^{3/4}\norm{\bv^{(0)}}_\infty\cdot\one +
	({4}/{3})\cdot\alpha_1\cdot \norm{\bv^{(0)}}_{\infty}\cdot\one,
	\]
	where $\alpha_1=\log(8|\cS||\cA|\delta^{-1})/m_1 \le 1$.
	Mover, since $ \bv(\bQ^{(i-1)})\le \bv^{(i)}$, we obtain
	\begin{align*}
	\bQ^* - \bQ^{(i)} &\le\gamma\bP \bv(\bQ^*) - \gamma \bP\bv(\bQ^{(i-1)}) + \bxi^{(i)} 
	\le \gamma\bP^{\pi^*}\bQ^* - \gamma \bP^{\pi^*}\bQ^{(i-1)} + \bxi^{(i)},
	\end{align*}
	where $\pi^*$ is an arbitrary optimal policy and we use the fact that $\max_{a}\bQ^{*}(s,a) = \bQ^*(s,\pi^*(s))$.
	This completes the proof of the lemma.
\end{proof}

\begin{proof}[Proof of Proposition~\ref{thm:alg 1}]
	Recall that we are able to sample a state from each $\bP_{s,a}$ with time $O(1)$.
	Let $\beta=(1-\gamma)^{-1}$,  $R=\lceil c_1\beta\ln[\beta u^{-1}]\rceil,
	m_1= c_2\beta^3u^{-2}\cdot\log(8|\cS||\cA|\delta^{-1})$ and $ m_2= c_3\beta^2\cdot\log[2R|\cS||\cA|\delta^{-1}]$ for some constants $c_1, c_2$ and $c_3$ required in Algorithm~\ref{alg-halfErr}.
	In the following proof, we set  $c_1, c_2, c_3$ to be sufficiently large but otherwise arbitrary absolute constants (e.g., $c_1 \ge 4, c_2\ge8192, c_3\ge128$).
	By Lemma~\ref{lemma:induction lemma}, with probability at least $1-2\delta$ for each $1\le i\le R$, we have $\bv^{(i-1)}\le \bv^{(i)}\le \cT_{\pi^{(i)}} \bv^{(i)}$, and $\bQ^{(i)}\le \br+\gamma\bP \bv^{(i)}$,
	\[
	\bQ^* - \bQ^{(i)} \le \gamma \bP^{\pi^*}\big[\bQ^* - \bQ^{(i-1)}\big] 
	+\bxi, %2\sqrt{{2\alpha_1\gamma^{-2}\wh{\sigma}_{\bv^{*}}}} + 2\rbr{ \sqrt{2\alpha_1}{u} + 8\alpha_1^{3/4}\norm{\bv^{(0)}}_\infty + ({2}/{3})\alpha_1\norm{\bv^{(0)}}_{\infty} + (1-\gamma){u}/8}\cdot\one,
	\]
	where
	\[
	\bxi\le {(1-\gamma){u}}/{C}\cdot \one + 
	C\sqrt{\alpha_1\bsigma_{\bv^*}} + C\alpha_1^{3/4}\norm{\bv^{(0)}}_\infty\cdot\one
	%2\sqrt{2\alpha_1\bsigma_{\bv^*}} +2\sqrt{2\alpha_1}{u}\cdot\one+ 16\alpha_1^{3/4}\norm{\bv^{(0)}}_\infty\cdot\one +
	%({4}/{3})\cdot\alpha_1\cdot \norm{\bv^{(0)}}_{\infty}\cdot\one,
	\]
	for  $\alpha_1=\log(8|\cS||\cA|\delta^{-1})/m_1$ and sufficiently large constant $C$.
	Solving the recursion, we obtain
	\begin{align*}
	\bQ^* - \bQ^{(R-1)} &\le \gamma^{R-1} \bP^{\pi^*}\big[\bQ^* - \bQ^{0}\big] 
	+\sum_{i=0}^{R-1}\gamma^{i}(\bP^{\pi^*})^{i}\bxi \\
	&\le \gamma^{R-1} \bP^{\pi^*}\big[\bQ^* - \bQ^{0}\big] + (I-\gamma\bP^{\pi^*})^{-1}\bxi.
	\end{align*}
	We first apply a na\"ive bound $\norm{\bP^{\pi^*}\big[\bQ^* - \bQ^{0}\big]}_{\infty}\le ({1-\gamma})^{-1}$.
	Hence
	\[
	\gamma^{R-1}\bP^{\pi^*}\big[\bQ^* - \bQ^{0}\big] \le \frac{{u}}{4}\cdot\one,
	\]
	where $R=\lceil(1-\gamma)^{-1}\ln[4(1-\gamma)^{-1}{u}^{-1}]\rceil + 1$.
	The next step is the key to the improvement in our analysis. We further apply the bound in Lemma~\ref{lemma:variance bound}, given by	%\lin{This is the main tweak.}
	\[
	%(\bI-\gamma\bP^{\pi^*})^{-1}\sqrt{\bsigma_{\bv^*}}\le 2\gamma^{-1}(1-\gamma)^{-1.5}\cdot\one.
	(\bI-\gamma\bP^{\pi^*})^{-1}\sqrt{\bsigma_{\bv^*}}\le \min(2\gamma^{-1}(1-\gamma)^{-1.5}, (1-\gamma)^{-2})\cdot\one \le 3(1-\gamma)^{-1.5}\cdot\one,
	\]
	where the last inequality follows since $\min(2\gamma^{-1},(1-\gamma)^{-1/2})\le 3$.
	With $\norm{(\bI-\gamma\bP^{\pi^*})^{-1}\one}_{\infty}\le (1-\gamma)^{-1}$ and $\norm{\bv^{(0)}}_\infty\le (1-\gamma)^{-1}$, we have,
	\begin{align*}
	(\bI-\gamma\bP^{\pi^*})^{-1}\bxi
	&\le \left[\frac{{u}}{8} + C'\sqrt{\frac{2\alpha_1}{\gamma^2(1-\gamma)^{3}}} +  
	%\frac{2\sqrt{2\alpha_1}{u}}{1-\gamma} + 
	%\frac{16\alpha_1^{3/4}}{(1-\gamma)^2} 
	C'\frac{\alpha_1^{3/4}}{(1-\gamma)^2} 
	%+\frac{4\alpha_1}{3(1-\gamma)^2}
	\right]\cdot\one\\
	%&\le \bigg[\frac{{u}}{8} + \frac{{u}}{16} + \frac{\gamma\sqrt{1-\gamma}{u}}{32} +  16\bigg(\frac{(1-\gamma)^3{u}^2}{32\cdot256\cdot(1-\gamma)^{8/3}}\bigg)^{3/4}  + \frac{4\gamma^2(1-\gamma){u}^2}{24\cdot 256}\bigg]\cdot\one\\
	&\le \bigg[\frac{{u}}{8} + \frac{{u}}{16}
	%\frac{\gamma\sqrt{1-\gamma}{u}}{C''} 
	+  \bigg(\frac{(1-\gamma)^3{u}^2}{C''\cdot(1-\gamma)^{8/3}}\bigg)^{3/4}  %\frac{4\gamma^2(1-\gamma){u}^2}{24\cdot 256}
	\bigg]\cdot\one\\
	&\le \frac{{u}}{4}\cdot \one,
	\end{align*}
	for some sufficiently large $C'$ and $C''$, 
	which depend on $c_1, c_2$ and $c_3$.
	%\[
	%\alpha_1=\frac{\log(8|\cS||\cA|\delta^{-1})}{m_1} = c_2^{-1}\beta^3u^{-2}%\le\frac{ (1-\gamma)^{3}{u}^2}{32\cdot 256}.
	%\]
	Since $\bv(\bQ^{(R-1)})\le\bv^{(R)}$,
	we have %\sidford{Can replace the maxes above and below with $\bv(\cdot)$?}
	\[
	\bv^* - \bv^{(R)} \le \bv^* -\bv(\bQ^{(R-1)}) \le\gamma^{R-1}\bP^{\pi^*}\big[\bQ^* - \bQ^{0}\big] + (\bI-\gamma\bP^{\pi^*})^{-1}\bxi \le \frac{{u}}{2}\cdot\one.
	\] 
	This completes the proof of the correctness.
	It remains to bound the time complexity.
	The initialization stage costs $O( m_1)$ time per $(s,a)$.
	Each iteration costs $O(m_2)$ time per $(s,a)$.
	We thus have the total time complexity as
	\[
	O( m_1 + R m_2)|\cS||\A| = O\bigg[{\frac{|\cS||\cA|}{(1-\gamma)^3}\cdot \log\frac{|\cS||\cA|}{\delta\cdot(1-\gamma)\cdot{u}}\cdot\bigg(\frac{1}{{u}^2}+\log\frac{1}{(1-\gamma)\cdot{u}}\bigg)}\bigg].
	\]
	Since $\log[(1-\gamma)^{-1}{u}^{-1}] = O(\log[(1-\gamma)^{-1}]{u}^{-2})$, we conclude the proof.
\end{proof}

%\subsection{Missing Proofs}

%\mw{It might be worth mentioning the monotonicity and contraction properties of the Bellman operations here. They are mentioned frequently in the later analysis.}\sidford{I assume this comment is now old? Delete if so.}

\subsection{Missing Analysis of Halving Errors}
\label{sec:half}

We refer in this section to Algorithm~\ref{alg-halfErr} as a subroutine \halfErr, which given an input MDP $\cM$ with a sampling oracle, an input value function $\bv^{(i)}$ and an input policy $\pi^{(i)}$, outputs an value function  $\bv^{(i+1)}$ and a policy $\pi^{(i+1)}$ such that, with high probability (over the new samples of the sampling oracle),
\[
\|\bQ^{(i+1)} - \bQ^*\|_{\infty}\le \|\bQ^{(i)} - \bQ^*\|_{\infty}/2 \quad\text{and}\quad 
\|\bv^{\pi^{(i+1)}}-\bv^*\|_{\infty}\le \|\bv^{\pi^{(i)}}-\bv^*\|_{\infty}/2.
\]
After $\log[\epsilon^{-1}(1-\gamma)^{-1}]$ calls of the subroutine \halfErr, the final output policy and value functions are $\epsilon$-close to the optimal ones with high probability. %\sidford{Don't we need more invariants then this for the policy?}\lin{that is for analysis. What we actually need is this one.}

We summarize our meta algorithm in Algorithm~\ref{alg-meta}.
Note that in the algorithm, each call of $\halfErr$ will draw new samples from the sampling oracle. 
These new samples guarantee the independence of successive improvements and also save space of the algorithm.
For instance, %it will become clear in subsequent sections that 
the algorithm \halfErr~only needs to use $O(|\cS||\cA|)$ words of memory instead of storing all the samples.  
%\sidford{Is there a reason to do this generality? Is it for finite horizon? Should make some argument for why regardless.}\lin{I don't understand the question though. This is just clarifying the space complexity.}
The guarantee of the algorithm is summarized in Proposition~\ref{prop:meta}.

\begin{proposition}
	\label{prop:meta}
	Let $\cM=(\cS, \cA, \br, \bP, \gamma)$ with a sampling oracle.
	Suppose \halfErr~is an algorithm that takes an input $\bv^{(i)}$ and an input policy $\pi^{(i)}$ and a number $u\in[0,(1-\gamma)^{-1}]$ satisfying $\bv^{*}-u\one\le \bv^{(i)}\le \bv^{\pi^{(i)}}$, halts in time $\tau$ and outputs a $\bv^{(i+1)}$ and a policy $\pi^{(i+1)}$ satisfying,
	\begin{align*}
	%\text{if } 
	%\|\bQ^{(i+1)} - \bQ^*\|_{\infty}\le \|\bQ^{(i)} - \bQ^*\|_{\infty}/2 \quad\text{and}\quad 
	\bv^{*}-\frac{u}{2}\cdot\one\le \bv^{(i+1)}\le \bv^{\pi^{(i+1)}} \le \bv^{*}.
	\end{align*}
	with probability at least $1-(1-\gamma)\cdot\epsilon\cdot\delta$ (over the randomness of the new samples given by the sampling oracle), then the meta algorithm described in Algorithm~\ref{alg-meta}, given input $\cM$ and the sampling oracle, halts in $\tau \cdot \log(\epsilon^{-1}\cdot (1-\gamma)^{-1})$ and outputs an policy $\pi^{(R)}$ such that
	\[
	\bv^{*}-{\epsilon}\cdot\one\le \bv^{(R)}\le \bv^{\pi^{(R)}} \le \bv^{*}
	\]
	with probability at least $1-\delta$ (over the randomness of all samples drawn from the sampling oracle).
	Moreover, if \halfErr~uses space $s$, then the meta algorithm uses space $s+O(|\cS||\cA|)$.
	If each call of \halfErr~ takes $m$ samples from the oracle, then the overall samples taken by Algorithm~\ref{alg-meta} is $m\cdot  \log(\epsilon^{-1}\cdot (1-\gamma)^{-1})$.
\end{proposition}
The proof of this proposition is a straightforward application of conditional probability.
%, thus postponed to the appendix (Section~\ref{sec:proof of main alg}).
%and postponed to the appendix (Section~\ref{sec:proof of main alg}).

\begin{proof}[Proof of Proposition~\ref{prop:meta}]
	The proof follows from a straightforward induction.
	For simplicity, denote $\beta = (1-\gamma)^{-1}$.
	In the meta-algorithm, the initialization is $\bv^{(0)}= {\bf0}$ and $\pi^{(0)}$ is an arbitrary policy.
	Thus $\bv^*- \beta\cdot \one\le \bv^{(0)} \le \bv^{\pi^{(0)}}$.
	By running the meta-algorithm, we obtain a sequence of value functions and policies:
	$\{\bv^{(i)}\}_{i=0}^R$ and $\{\pi^{(i)}\}_{i=0}^R$.
	Since each call of the \halfErr~uses new samples from the oracle, the sequence of value functions and policies satisfies strong Markov property (given $(\bv^{(i)}, \pi^{(i)})$, $(\bv^{(i+1)}, \pi^{(i+1)})$ is independent with $\{(\bv^{(j)}, \pi^{(j)})\}_{j=0}^{i-1}$).
	Thus
	\begin{align*}
	\Pr&\bigr[\bv^*- 2^{-R}\beta\cdot \one\le \bv^{(R)} \le \bv^{\pi^{(R)}}\bigr]
	\\
	&\ge \prod_{i=1}^{R}\Pr\bigr[\bv^*- 2^{-i}\beta\cdot \one\le \bv^{(i)} \le \bv^{\pi^{(i)}} \bigr|\bv^*- 2^{-i+1}\beta\cdot \one\le \bv^{(i-1)} \le \bv^{\pi^{(i-1)}}\bigr] \\
	& \ge 1-\delta.
	\end{align*}
	Since $2^{-R}(1-\gamma)^{-1}\le \epsilon$, we conclude the proof.
\end{proof}

\begin{proof}[Proof of Theorem~\ref{thm:dmdp1}]
	Our algorithm is simply plugging in Algorithm~\ref{alg-halfErr} as the \halfErr~subroutine in Algorithm~\ref{alg-meta}. 
	The correctness is guaranteed by Proposition~\ref{prop:meta} and Proposition~\ref{thm:alg 1}.
	The running time guarantee follows from a straightforward calculation.  
\end{proof}

\section{Extension to Finite Horizon}
\label{sec:finite_horizon}

In this section we show how to apply similar techniques to achieve improved sample complexities for solving finite Horizon MDPs given a generative model and we prove that the sample complexity we achieve is optimal up to logarithmic factors. 

The finite horizon problem is to compute an optimal non-stationary policy over a fixed time horizon $H$, i.e. a policy of the form $\pi(s, h)$ for $s \in S$ and $h \in \{0, \ldots H\}$), where the reward is the expected cumulative (un-discounted) reward for following this policy. In classic value iteration, this is typically done using a backward recursion from time $H, H-1, \ldots 0$. We show how to use the ideas in this paper to solve for an $\epsilon$-approximate policy.  
As we have shown in the discounted case, it is suffice to show an algorithm that decrease the error of the value at each stage by half.
Our algorihtm is presented in Algorithm~\ref{algH-halfErr}.

To analyze the algorithm, we first provide an analogous lemma of Lemma~\ref{lemma: emprical mean},
\begin{lemma}[Empirical Estimation Error]
	\label{lemma: emprical mean h}
	Let $\wt{\bw}_h$ and $\wh{\bsigma}_h$ be computed in Line \ref{alg1: computeH w} of Algorithm \ref{algH-halfErr}.
	Recall that $\wt{\bw}_h$ and $\wh{\bsigma}_h$ are empirical estimates of $\bP\bv_h$ and $\bsigma_{\bv_h}=\bP\bv_h^2 - (\bP\bv_h)^2$ using $m_1$ samples per $(s,a)$ pair.
	Then with probability at least $1-\delta$, for $L \defeq \log(8|\cS||\cA|\delta^{-1})$ and every $h=1,2,\ldots, H$, we have 
	{\small
		\begin{align}
		\label{eqn:estimateH pv}
		%\forall (s,a)\in \cS\times \cA:\quad	
		\big|{\wt{\bw}_h - \bP^\top\bv_h^{(0)}}\big|\le \sqrt{{2m_1^{-1}\bsigma_{\bv_h^{(0)}}\cdot{L}}} + {2(3m_1)^{-1}\norm{\bv_h^{(0)}}_{\infty} {L}} 
		\end{align}
	}
	and
	{\small
		\begin{align}
		\label{eqn:estimateH sigma}
		\forall (s,a)\in \cS\times \cA:\quad
		\big|\wh{\bsigma}_h(s,a) - \bsigma_{\bv_h^{(0)}}(s, a)\big| \le 4\norm{\bv_h^{(0)}}_{\infty}^2 \cdot \sqrt{{2m_1^{-1}{L}}}.
		\end{align}
	}
\end{lemma}
\begin{proof}
	The proof of this lemma is identical to that of Lemma~\ref{lemma: emprical mean}.
\end{proof}
An analogous lemma to Lemma~\ref{lemma:bounds on g} is also presented here.
\begin{lemma}
	\label{lemma:boundsH on g}
	Let $\bg_h^{(i)}$ be the estimate of $\bP\big[\bv_h^{(i)}  - \bv_h^{(0)}\big]$ defined in Line~\ref{alg1: computeH g} of Algorithm~\ref{algH-halfErr}. 
	%Let $\delta' = \delta/R$.\sidford{I would remove the definition of $\delta'$ and just propagate the $/R$ where needed, don't think it is worth the extra definition.}
	%for some integer $R> 0$.
	%Note that $\bg^{(i)}$ is a $[(1-\gamma)\epsilon/8]$-shifted empirical estimation with $m_2=128\cdot(1-\gamma)^{-2}\cdot\log(2|\cS||\cA|/\delta')$ samples per $(s,a)$,
	Then conditioning on the event that $\norm{\bv_h^{(i)} - \bv_h^{(0)}}_{\infty}\le 2\epsilon$, with probability at least $1-\delta/H$, 
	\[
	\bP  \big[\bv_h^{(i)}  - \bv_h^{(0)}\big] - \frac{\epsilon}{4H}\cdot\one\le \bg_h^{(i)} \le \bP  \big[\bv_h^{(i)}  - \bv_h^{(0)}\big]
	\]
	provided appropriately chosen constants in Algorithm~\ref{algH-halfErr}.
\end{lemma}
\begin{proof}
	The proof of this lemma is identical to that of Lemma~\ref{lemma:bounds on g} except that $(1-\gamma)^{-1}$ is replaced with $H$.
\end{proof}

Similarly, we can show the following improvement lemma.

\begin{lemma}
	\label{lemma:inductionH lemma}
	Let $\bQ_h$ be the estimated $Q$-function of $\bv_{h+1}$ in Line~\ref{alg: Hq-func} of Algorithm \ref{algH-halfErr}.
	Let $\bQ_h^* = \br + \bP_h \bv^*_{h+1}$ be the optimal $Q$-function of the DMDP.
	Let $\pi(\cdot, h)$ and $\bv_h$ be estimated in iteration $h$, as defined in Line~\ref{alg: Hv1} and \ref{alg: Hv2}.
	Let $\pi^*$ be an optimal policy for the DMDP. 
	For a policy $\pi$, let $\bP_h^{\pi}\bQ\in\RR^{\cS\times\cA}$ be defined as $(\bP_h^{\pi}\bQ)(s,a) = \sum_{s'\in\cS}\bP_{s,a}(s')\bQ(s',\pi(s',h))$.
	Suppose for all $h\in[H-1]$, $\bv_h^{(0)}\le  \cT_{\pi^{(0)}(\cdot, h)} \bv_{h+1}^{(0)}$.
	Let $\bv_{H+1}\defeq \bf{0}$ and $\bQ_{H+1}\defeq 0$.
	Then, with probability at least $1- 2\delta$, for all $1\le h \le H$, 
	$\bv_h^{(0)}\le \bv_h\le \cT_{\pi(\cdot, h)} \bv_{h+1}
	\le \bv_{h}^*$, $\bQ_h\le \br+\bP_h \bv_{h+1}$ and 
	{\small
		\[
		\bQ^*_h - \bQ_h \le  \bP^{\pi^*}_h\big[\bQ^*_{h+1} - \bQ_{h+1}\big] 
		+\bxi_h, %2\sqrt{{2\alpha_1\gamma^{-2}\wh{\sigma}_{\bv^{*}}}} + 2\rbr{ \sqrt{2\alpha_1}\epsilon + 8\alpha_1^{3/4}\norm{\bv^{(0)}}_\infty + ({2}/{3})\alpha_1\norm{\bv^{(0)}}_{\infty} + (1-\gamma)\epsilon/8}\cdot\one,
		\]
	}
	where the error vector $\bxi_h$ satisfies %\sidford{add "$\bxi$" satisfies?}
	\[
	{\bf0}\le \bxi_h\le {8H^{-1}u}\cdot \one + 2\sqrt{2\alpha_1\bsigma_{\bv^*_{h+1}}} +2\sqrt{2\alpha_1}u\cdot\one+ 16\alpha_1^{3/4}\norm{\bv^{(0)}_{h+1}}_\infty\cdot\one +
	({4}/{3})\cdot\alpha_1\cdot \norm{\bv^{(0)}_{h+1}}_{\infty}\cdot\one,
	\]
	and  $\alpha_1=\log(8|\cS||\cA|H\delta^{-1})/m_1$. 
\end{lemma}

\begin{proof}[Proof of Lemma~\ref{lemma:inductionH lemma}]
	%We prove by induction on $i$.
	By Lemma~\ref{lemma: emprical mean}, for any $h=1, 2, \ldots, H$, with probability at least $1-\delta/H$,
	\begin{align*}
	|{\wt{\bw}_h - \bP\bv_{h+1}}| \le \sqrt{2\alpha_1\bsigma_{\bv_{h+1}^{(0)}}} +
	\frac{2}{3}\cdot\alpha_1\cdot \norm{\bv_{h+1}^{(0)}}_{\infty}\cdot\one,
	\end{align*}
	and
	\begin{align}
	\label{eqn:varH}
	\big|\wh{\bsigma}_{h+1}- \bsigma_{\bv^{(0)}_{h+1}}\big| \le 4\norm{\bv^{(0)}_{h+1}}_{\infty}^2 \cdot \sqrt{2\alpha_1} \cdot\one,
	\end{align}
	which we condition on. 
	We have
	\[
	|{\wt{\bw}_h - \bP\bv_{h+1}^{(0)}}| \le  \sqrt{2\alpha_1\wh{\bsigma}_{h+1}} + (4\alpha_1^{3/4}\norm{\bv^{(0)}_{h+1}}_\infty +
	\frac{2}{3}\cdot\alpha_1\cdot \norm{\bv^{(0)}_{h+1}}_{\infty})\one.
	\]
	Thus 
	\begin{align}
	\label{eqn:upper Hw}
	\bw_h = \wt{\bw}_h - \sqrt{2\alpha_1\wh{\bsigma}_{h+1}} - 4\alpha_1^{3/4}\norm{\bv^{(0)}_{h+1}}_\infty\one -
	\frac{2}{3}\cdot\alpha_1\cdot \norm{\bv^{(0)}_{h+1}}_{\infty}\one 
	\le  \bP\bv^{(0)}_{h+1},
	\end{align}
	and
	\begin{align*}
	\bw_{h}\ge \bP\bv^{(0)}_{h+1} - 2\sqrt{2\alpha_1\wh{\bsigma}_{h+1}} - (8\alpha_1^{3/4}\norm{\bv^{(0)}_{h+1}}_\infty +
	\frac{4}{3}\cdot\alpha_1\cdot \norm{\bv^{(0)}_{h+1}}_{\infty})\one
	. 
	\end{align*}
	By \eqref{eqn:var} and Lemma~\ref{lemma: variance triangle}, we have
	\[
	\sqrt{\wh{\bsigma}_{h+1}} \le \sqrt{\bsigma_{\bv^{(0)}_{h+1}}} + 2 \norm{\bv^{(0)}_{h+1}}_{\infty} (2\alpha)^{1/4}\one
	\le \sqrt{\bsigma_{\bv^{*}_{h+1}}} +  \epsilon \one+ 2 \norm{\bv_{h+1}^{(0)}}_{\infty} (2\alpha)^{1/4}\one.
	\]
	we have
	\begin{align}
	\label{eqn:lower Hw}
	\bw_h\ge \bP\bv_{h+1}^{(0)} - 2\sqrt{2\alpha_1\bsigma_{\bv_{h+1}^*}} -2\sqrt{2\alpha_1}\epsilon\one- 16\alpha_1^{3/4}\norm{\bv_{h+1}^{(0)}}_\infty\one -
	\frac{4}{3}\cdot\alpha_1\cdot \norm{\bv_{h+1}^{(0)}}_{\infty}\one 
	\end{align}
	
	%Next, by definition of $\bv^{(i)}$ (Line~\ref{alg: v1} and \ref{alg: v2}), we have
	%\[
	%\bv^{(0)}\le \bv^{(1)}.
	%\]
	For the rest of the proof, we condition on the event that \eqref{eqn:upper Hw} and \eqref{eqn:lower Hw} hold for all $h=1, 2, \ldots, H$, which happens with probability at least $1-\delta$.
	Denote $\bv_{H+1}^* = \bv_{H+1} = \bv_{H+1}^{(0)} ={\bf 0}$.
	Thus we have $\bv_{H+1}^{(0)}\le \bv_{H+1} \le \bv_{H+1}^{*}$.
	Next we prove the lemma by induction on $h$.
	Assume for some $h$, %and assume for each $0\le k\le i-1$, 
	with probability at least $1-(h-1)\delta/H$ the following holds, for all $h' = h+1, h+2, \ldots, H,$
	\[
	 \bv^{(0)}_{h'} \le  \bv_{h'} \le \bv^{*}_{h'},
	\]
	which we condition on.
	Next we show that the lemma statement holds for $h$ as well.
	By definition of $\bv_h$ (Line~\ref{alg: v1} and \ref{alg: v2}),
	\[
	\bv_{h}^{(0)}\le \bv_h \quad\text{and}\quad  \bv(\bQ_{h})\le \bv_h.
	\]
	Furthermore, since $\bv^{(0)}_{h+1}\le \bv^*_{h+1} \le \bv^{(0)}_{h+1} + u\one$ %\bv^{(1)} \le \ldots \le \bv^{(i-1)}\le \cT_{\pi^{i-1}}\bv^{(i-1)} \le \cT\bv^{(i-1)}\le \cT^{\infty}\bv^{(i-1)} = \bv^*$, 
	we have
	\[
	\bv_{h+1}^*-\bv_{h+1} \le \bv^{*}_{h+1} - \bv^{(0)}_{h+1} \le u\one.
	\]
	By Lemma~\ref{lemma:bounds on g}, we have, with probability at least $1-\delta'$
	\begin{align}
	\label{eqn:bound Hg}
	\bP  \big[\bv_{h+1}  - \bv_{h+1}^{(0)}\big] - \frac{u}{8H}\cdot\one\le \bg_h \le \bP  \big[\bv_{h+1}  - \bv^{(0)}_{h+1}\big],
	\end{align}
	which we condition on for the rest of the proof.
	Thus we have
	\[
	\bQ_h = \br + (\bw_h + \bg_h) \le \br + \bP\bv^{(0)}_{h+1} + \bP\bv_{h+1} - \bP\bv^{(0)}_{h+1} = \br + \bP \bv_{h+1} \le \bQ_h^*.
	\]
	To show $\bv_h\le \cT_{\pi(\cdot, h)} \bv_{h+1}$, we notice that if for some $s$,  $\pi(s,h)\neq \pi^{(0)}(s,h)$, then, 
	\[\bv_h(s) \le \br(s, \pi(s, h))+\bP_{s, \pi(s,h)}^\top\bv_{h+1} = \cT_{\pi(\cdot, h)} \bv_{h+1}.
	\]
	On the other hand, if $\pi(s, h)= \pi^{(0)}(s, h)$, then 
	\[
	\forall s\in \cS:\quad \bv_h(s) = \bv_{h}^{(0)}(s) \le (\cT_{\pi^{(0)}(\cdot, h)} \bv_{h+1}^{(0)})(s) \le (\cT_{\pi^{(0)}(\cdot, h)} \bv_{h+1})(s) = (\cT_{\pi(\cdot, h)} \bv_{h+1})(s).
	\]
	This completes the induction step.
	Lastly, combining \eqref{eqn:lower Hw} and \eqref{eqn:bound Hg}, we have
	\begin{align*}
	\bQ^*_h - \bQ_h &= \bQ^{*}_h - \br - (\bw_h + \bg_h) 
	= \bP\bv(\bQ_{h+1}^*) -  (\bw_h + \bg_h) \\
	&=\bP\bv(\bQ^*_{h+1}) -  \bP(\bv_{h+1} -\bv_{h+1}^{(0)}) - \bP \bv_{h+1} + \bxi_h\\
	&=\bP\bv(\bQ^*_{h+1}) -  \bP\bv_{h+1} + \bxi_h,
	\end{align*}
	where
	\[
	\bxi_h \le {H^{-1}u}/{8}\cdot \one + 2\sqrt{2\alpha_1\bsigma_{\bv_{h+1}^*}} +2\sqrt{2\alpha_1}u\cdot\one+ 16\alpha_1^{3/4}\norm{\bv_{h+1}^{(0)}}_\infty\cdot\one +
	({4}/{3})\cdot\alpha_1\cdot \norm{\bv_{h+1}^{(0)}}_{\infty}\cdot\one,
	\]
	where $\alpha_1=\log(8|\cS||\cA|\delta^{-1})/m_1$.
	Mover, since $ \bv(\bQ_{h+1})\le \bv_{h+1}$, we obtain
	\begin{align*}
	\bQ_h^* - \bQ_h &\le\bP \bv(\bQ_{h+1}^*) - \bP\bv(\bQ_{h+1}) + \bxi_h
	\le \bP^{\pi^*}_h\bQ^*_{h+1} - \bP^{\pi^*}_h\bQ_{h+1} + \bxi_h,
	\end{align*}
	where $\pi^*$ is an arbitrary optimal policy and we use the fact that $\max_{a}\bQ_{h}^{*}(s,a) = \bQ_{h}^*(s,\pi^*(s, h))$.
	This completes the proof of the lemma.
\end{proof}

Furthermore, we show an analogous lemma of Lemma~\ref{lemma:variance bound}.
\begin{lemma}[Upper Bound on Variance]
	\label{lemma:varianceH bound}
	For any $\pi$, we have
	\[
	\bigg\|\sum_{h'=h}^{H-1}\bigg(\prod_{i=h+1}^{h'}\bP_i^{\pi}\bigg)\sqrt{{\bsigma}_{\bv_{h'+1}^{\pi}}}\bigg\|_{\infty}^2 \le H^{3/2}.
	\]
\end{lemma}
\begin{proof}
First, by Cauchy-Swartz inequality, we have
\begin{align*}
	\sum_{h'=h}^{H-1}\bigg(\prod_{i=h+1}^{h'}\bP_i^{\pi}\bigg) \sqrt{{\bsigma}_{\bv_{h'+1}^{\pi}}}\le \sqrt{H\sum_{h'=h}^{H-1}\bigg(\prod_{i=h+1}^{h'}\bP_i^{\pi}\bigg){\bsigma}_{\bv_{h'+1}^{\pi}}}.
\end{align*}
Next, by a similar argument of the proof of  Lemma~\ref{lem:var_bell}, we can show that 
\[
\bigg[\sum_{h'=h}^{H-1}\bigg(\prod_{i=h+1}^{h'}\bP_i^{\pi}\bigg){\bsigma}_{\bv_{h'+1}^{\pi}}\bigg] (s) = \var\bigg[\sum_{t=h}^H r(s^t, \pi(s^t, t))\bigg|s^h = s\bigg] \le H^2.
\]
This completes the proof.
\end{proof}
We are now ready to present the guarantee of the algorithm ~\ref{algH-halfErr}.

\begin{proposition}
	\label{thm:algH 1}
	%Suppose we are able to sample a state from each $\bP_{s,a}$ with time $\tau_s$.
	%Let $\beta=(1-\gamma)^{-1}$,  $R=\lceil c_1\beta\ln[\beta u^{-1}]\rceil,
	%m_1= c_2\gamma^{-2}\beta^3u^{-2}\cdot\log(8|\cS||\cA|\delta^{-1})$ and $ m_2= c_3\beta^2\cdot\log[2R|\cS||\cA|\delta^{-1}]$ for some constants $c_1, c_2$ and $c_3$.
	On an input value vectors $\bv_1^{(0)}, \bv_2^{(0)}, \ldots, \bv_H^{(0)}$, policy $\pi^{(0)}$, and parameters $u\in(0, \beta], \delta\in(0,1)$ such that $\bv_h^{(0)}\le \cT_{\pi^{(0)}(\cdot,h)}\bv_{h+1}^{(0)}$ for all $h\in[H-1]$, and $\bv_h^{(0)}\le \bv_h^* \le \bv_h^{(0)} + u \one$,
	Algorithm~\ref{algH-halfErr} halts in time
	$
	O[u^{-2}\cdot H^4|\cS||\cA|\cdot \log(|\cS||\cA\delta^{-1}H u^{-1})]
	%O[\gamma^{-2}u^{-2}\beta\cdot|\cS||\cA|\cdot \log(|\cS||\cA\delta^{-1}\beta u^{-1})]
	$
	and outputs $\bv_1, \bv_2, \ldots, \bv_{H}$ and $\pi:\cS\times[H]\rightarrow \cA$ such that 
	\begin{align*}
	\forall h\in [H]: \quad \bv_h \le \cT_{\pi(\cdot, h)}(\bv_{h+1})\quad\text{and}\quad {\bf 0}\le \bv_h^*-\bv_h \le (u/2)\cdot\one
	\end{align*} with probability at least $1-\delta$, provided appropriately chosen constants, $c_1, c_2$ and $c_3$, in Algorithm~\ref{algH-halfErr}.
	Moreover, the algorithm uses $
	O[u^{-2}\cdot H^3|\cS||\cA|\cdot \log(|\cS||\cA\delta^{-1}H u^{-1})]
	$ samples from the sampling oracle.
\end{proposition}

\begin{proof}[Proof of Proposition~\ref{thm:algH 1}]
	Recall that we are able to sample a state from each $\bP_{s,a}$ with time $O(1)$.
	Let $R=\lceil c_1H\ln[H u^{-1}]\rceil,
	m_1= c_2H^3u^{-2}\cdot\log(8|\cS||\cA|\delta^{-1})$ and $ m_2= c_3H^2\cdot\log[2R|\cS||\cA|\delta^{-1}]$ for some constants $c_1, c_2$ and $c_3$ required in Algorithm~\ref{alg-halfErr}.
	In the following proof, we set  $c_1 = 4, c_2=8192, c_3=128$.
	By Lemma~\ref{lemma:induction lemma}, with probability at least $1-2\delta$ for each $1\le h\le H$, we have $\bv_h^{(0)}\le \bv_h\le \cT_{\pi(\cdot, h)} \bv_h$, and $\bQ_h\le \br+\bP \bv_{h+1}$,
	\[
	\bQ_h^* - \bQ_h \le  \bP_h^{\pi^*}\big[\bQ_{h+1}^* - \bQ_{h+1}\big] 
	+\bxi_{h}, %2\sqrt{{2\alpha_1\gamma^{-2}\wh{\sigma}_{\bv^{*}}}} + 2\rbr{ \sqrt{2\alpha_1}{u} + 8\alpha_1^{3/4}\norm{\bv^{(0)}}_\infty + ({2}/{3})\alpha_1\norm{\bv^{(0)}}_{\infty} + (1-\gamma){u}/8}\cdot\one,
	\]
	where
	\[
	\bxi_h\le {H^{-1}{u}}/{8}\cdot \one + 2\sqrt{2\alpha_1\bsigma_{\bv_{h+1}^*}} +2\sqrt{2\alpha_1}{u}\cdot\one+ 16\alpha_1^{3/4}\norm{\bv_{h+1}^{(0)}}_\infty\cdot\one +
	({4}/{3})\cdot\alpha_1\cdot \norm{\bv_{h+1}^{(0)}}_{\infty}\cdot\one,
	\]
	and  $\alpha_1=\log(8|\cS||\cA|\delta^{-1})/m_1$.
	Notice that $\bv_H^{(0)} = \bv_H^* = \bv(\br)$, thus the $\bv_H-\bv_H^* = \bf{0}$. 
	Solving the recursion, we obtain
	\begin{align*}
	\bQ_h^* - \bQ_h &\le 
	\sum_{h'=h}^{H-1}\bigg(\prod_{i=h+1}^{h'}\bP_i^{\pi^*}\bigg)\bxi_{h'}.
	\end{align*}
	The next step is the key to the improvement in our analysis. We further apply the bound in Lemma~\ref{lemma:variance bound}, given by	%\lin{This is the main tweak.}
	\[
	%(\bI-\gamma\bP^{\pi^*})^{-1}\sqrt{\bsigma_{\bv^*}}\le 2\gamma^{-1}(1-\gamma)^{-1.5}\cdot\one.
	\sum_{h'=h}^{H-1}\bigg(\prod_{i=h+1}^{h'}\bP_i^{\pi^*}\bigg)\sqrt{\bsigma_{\bv_{h'+1}^*}}\le H^{3/2}\cdot\one.
	\]
	With $\norm{	\sum_{h'=h}^{H-1}\prod_{i=h+1}^{h'}\bP_i^{\pi^*}\one}_{\infty}\le H -h + 1$ and $\norm{\bv^{(0)}_h}_\infty\le H$, we have,
	\begin{align*}
	\sum_{h'=h}^{H-1}\bigg(\prod_{i=h+1}^{h'}\bP_i^{\pi^*}\bigg)\bxi_{h'}
	&\le \left[\frac{{u}}{8} + 4\sqrt{2\alpha_1H^{3}} + {2H\sqrt{2\alpha_1}{u}} + {16H^2\alpha_1^{3/4}} +\frac{4\alpha_1H^2}{3}\right]\one\\
	&\le \bigg[\frac{{u}}{8} + \frac{{u}}{16} + \frac{\sqrt{H^{-1}}{u}}{32} +  16\bigg(\frac{H^{-3}{u}^2}{32\cdot256\cdot(H)^{-8/3}}\bigg)^{3/4}  + \frac{4H^{-1}{u}^2}{24\cdot 256}\bigg]\cdot\one\\
	&\le \frac{{u}}{4}\cdot \one,
	\end{align*}
	provided
	\[
	\alpha_1=\frac{\log(8|\cS||\cA|\delta^{-1})}{m_1} = c_2^{-1}H^3u^{-2}\le\frac{ H^{-3}{u}^2}{32\cdot 256}.
	\]
	Since $\bv(\bQ_h)\le\bv_h$,
	we have %\sidford{Can replace the maxes above and below with $\bv(\cdot)$?}
	\[
	\bv^*_h - \bv_h \le \bv^* -\bv(\bQ_h) \le\sum_{h'=h}^{H-1}\bigg(\prod_{i=h+1}^{h'}\bP_i^{\pi^*}\bigg)\bxi_{h'} \le \frac{{u}}{2}\cdot\one.
	\] 
	This completes the proof of the correctness.
	It remains to bound the time complexity.
	The initialization stage costs $O( m_1)$ time per $(s,a)$ per stage $h$.
	Each iteration costs $O(m_2)$ time per $(s,a)$.
	We thus have the total time complexity as
	\[
	O( Hm_1  + H m_2)|\cS||\cA| = O\bigg[{H^4\cdot |\cS||\cA|\cdot \log\frac{H|\cS||\cA|}{\delta\cdot{u}}\cdot\frac{1}{{u}^2}}\bigg].
	\]
	The total number of samples used is 
	\[
	O( m_1  + H m_2)|\cS||\cA| = O\bigg[{H^3\cdot |\cS||\cA|\cdot \log\frac{H|\cS||\cA|}{\delta\cdot{u}}\cdot\frac{1}{{u}^2}}\bigg].
	\]
	This completes the proof.
\end{proof}

%\sidford{Why is the algorithm not more in the beginning of the section?}\lin{The meta algorithm?}

\begin{algorithm}\caption{FiniteHorizonRandomQVI\label{algH-halfErr}}
	\begin{algorithmic}[1]
		\State 
		\textbf{Input:} $\cM=(\cS, \cA, \br, \bP)$ with a sampling oracle,
		$\bv^{(0)}_1, \bv^{(0)}_2, \ldots, \bv^{(0)}_H, \pi^{(0)}: \cS\times[H]\rightarrow\cA, u, \delta\in(0,1)$;
		\State
		\emph{\textbackslash \textbackslash $u$ is the initial error, $\pi^{(0)}$ is the input policy, and $\delta$ is the error probability}
		\State\textbf{Output:} $\bv_1, \bv_2, \ldots, \bv_H, \pi$
		%\State
		%\Comment{$c_1$ is a constant}
		%\State Let $m_1 \gets{\text{?}}$, \\
%		\quad $m_2\gets {\epsilon^{-2}\log[2|\cS||\cA|\delta^{-1}]}$
%		\State Let $m_1 \gets{c_2H^3u^{-2}{\log(8|\cS||\cA|\delta^{-1})} }{}$, \\
%		\quad $m_2\gets {c_3H^{2}\log[2R|\cS||\cA|\delta^{-1}]}$ and\\
%		\quad $\alpha_1\gets{m_1}^{-1}{\log(8|\cS||\cA|\delta^{-1})}$;
%		\Comment{$c_2$ and $c_3$ are constants} %and
		\State
		\State\textbf{INITIALIZATION:} 
		%\emph{\textbackslash \textbackslash $c_1$ is a constant}
		\State Let $m_1 \gets{c_1H^3u^{-2}{\log(8|\cS||\cA|\delta^{-1})} }{}$ for constant $c_1$;
		\State Let $m_2\gets {c_2H^{2}\log[2H|\cS||\cA|\delta^{-1}]}$ for constant $c_2$;
		\State Let $\alpha_1\gets m_1^{-1}{\log(8|\cS||\cA|\delta^{-1})}$;
		\State
		For each  $(s, a)\in \cS\times\cA$,
		sample independent samples $s_{s,a}^{(1)}, s_{s,a}^{(2)}, \ldots, s_{s,a}^{(m_1)}$ from $\bP_{s,a}$;
		\State
		Initialize $\bw_h=\wt{\bw}_h = \wh{\bsigma}_h=\bQ^{(0)}_h \gets {\bf0}_{\cS \times \cA}$ for all $h\in[H]$, and $i\gets 0$;	
		\label{alg1: computeH w} 
		\State Denote $\bv_{H+1}\gets \bf{0}$ and $\bQ_{H+1}\gets \bf{0}$
		\For{each $(s, a)\in \cS\times\cA$, $h\in [H]$} 
		\State \emph{\textbackslash \textbackslash Compute empirical estimates of $\bP_{s,a}^{\top}\bv_h^{(0)}$ and $\bsigma_{\bv_{h}^{(0)}}(s,a)$}
		\State 
		Let $\wt{\bw}_h(s,a) \gets \frac{1}{m_1} %m_1^{-1}
		\sum_{j=1}^{m_1} \bv^{(0)}_h(s_{s,a}^{(j)})$ 
		\State 
		%Let $\wh{\bsigma}(s,a)\gets \gamma^2\big[{m_1^{-1}}\sum_{j=1}^{m_1}(\bv^{(0)})^2(s_{s,a}^{(j)}) - \wt{\bw}^2(s,a)\big]$ 
		Let $\wh{\bsigma}_h(s,a)\gets \frac{1}{m_1} \sum_{j=1}^{m_1}(\bv^{(0)}_h)^2(s_{s,a}^{(j)}) - \wt{\bw}_h^2(s,a)$ 
		\State
		
		\State \emph{\textbackslash \textbackslash Shift the empirical estimate to have one-sided error} 
		\State 	
		$\bw_h(s, a) \gets \wt{\bw}_h(s,a) - \sqrt{2\alpha_1\wh{\bsigma}_h(s,a)} - 4\alpha_1^{3/4}\norm{\bv^{(0)}_h}_\infty - (2/3)\alpha_1\norm{\bv^{(0)}_h}_{\infty}$
		%\State
		
		%\State \emph{\textbackslash \textbackslash Compute coarse estimate of the  $Q$-function}
		% \big({\frac{2\wh{\bsigma}(s,a)\cdot\log{(8|\cS||\cA|\delta^{-1})}}{m_1}}\big)^{1/2} - \frac{2\norm{\bv^{(0)}}_\infty\cdot\log{(8|\cS||\cA|\delta^{-1})}}{m_1}$, and
		%\State
		%$\bQ^{(0)}_h(s,a) \gets \br(s,a) + \gamma \bw(s,a)_h$
		\EndFor
		\State Let $\bv_{H+1}\gets \bf{0}$ and $\bQ_{H+1}\gets \bf{0}$.
		\State
		%\State\textbf{Repeat:}
		\State\textbf{REPEAT:} \emph{\textbackslash \textbackslash successively improve}
		\For{$h=H, H-1$ to $1$}
		\State \emph{\textbackslash \textbackslash Compute $\bP_{s,a}^\top  \big[\bv_h  - \bv_h^{(0)}\big]$ with one-sided error}
		\State\label{alg: Hv1} Let $\wt{\bv}_h \gets {\bv}_h \gets \bv(\bQ_{h+1})$, $\wt{\pi}(\cdot, h)\gets {\pi}(\cdot, h)\gets \pi(\bQ_{h+1})$, $\bv_h\gets \wt{\bv}_h$;
		\State\label{alg: Hv2} For each $s\in \cS$, if $\wt{\bv}_{h}(s)\le \bv^{(0)}_{h}(s)$, then 
		$\bv_h(s)\gets \bv^{(0)}_{h}(s)$ and $\pi(s,h)\gets \pi^{(0)}(s,h)$;
		\State For each $(s, a)\in \cS\times\cA$,
		sample independent samples $\wt{s}_{s,a}^{(1)}, \wt{s}_{s,a}^{(2)}, \ldots, \wt{s}_{s,a}^{(m_2)}$ from $\bP_{s,a}$;
		\State  \label{alg1: computeH g} Let $\bg_{h}(s,a)\gets {m_2^{-1}}\sum_{j=1}^{m_2} \big[\bv_{h}(\wt{s}_{s,a}^{(j)}) - \bv_h^{(0)}(\wt{s}_{s,a}^{(j)}) \big]- H^{-1}u/8$;
		\State
		\State 
		\emph{\textbackslash \textbackslash   Improve $\bQ_h$:}
		\State \label{alg: Hq-func} $\bQ_h\gets \br + \bw_h+\bg_h$; 
		\EndFor
		\State \textbf{return} $\bv_1, \bv_2, \ldots, \bv_H, \pi$.
%		\State
%		\textbf{For} each  $(s, a)\in \cS\times\cA$ \textbf{do}
%		sample independent samples $s_{s,a}^{(1)}, s_{s,a}^{(2)}, \ldots, s_{s,a}^{(m_1)}$ from $\bP_{s,a}$;
%		
%		\State \textbf{For} $h=1, 2\ldots, h$, \textbf{do} compute $\wt{\bw}_{h}(s,a) \gets \frac{1}{m_1} %m_1^{-1}
%\sum_{j=1}^{m_1} \bv_{h}^{(0)}(s_{s,a}^{(j)})$; 
%		
%		\For{$h = H, H-1,\ldots, 1$}
%			
%		\EndFor
%		\State \label{algH: q-func} $\bQ\gets \br + \wt{\bw}$;
%
%		\State\label{algH: v1} Let $\bv^{(2)} \gets \max \bQ$, $\pi^{(2)}\gets \arg\max\bQ$;
%		
%		\State \textbf{return} $\bv^{(2)}, \pi^{(2)}$.
%
%
%%		\State  \label{algH: compute g} Let $\bg(s,a)\gets {m_2^{-1}}\sum_{j=1}^{m_2} [\bv^{(1)}(s_{s,a}^{(j)}) - \bv^{(0)}(s_{s,a}^{(j)})]$;
%%		
%%		\State \label{algH: q-func} $\bQ\gets \br + \wt{\bw} + \bg$;
%%		
%%		
%%		\State\label{algH: v1} Let $\bv^{(2)} \gets \max \bQ$, $\pi^{(2)}\gets \arg\max\bQ$;
%%
%%		\State \textbf{return} $\bv^{(2)}, \pi^{(2)}$.
	\end{algorithmic}
\end{algorithm}

We can then use our meta-algorithm and obtain the following theorem.

\begin{theorem}
	Let $\cM=(\cS, \cA, \bP, \br, H)$ be a $H$-MDP with a sampling oracle. 
	Suppose we can sample a state from each probability vector $\bP_{s,a}$ within time $O(1)$.	Then there exists an algorithm that runs in time 
	\[
	O\bigg[{\frac{1}{\epsilon^2}}\cdot{H^4|\cS||\cA|}\cdot \log\frac{H|\cS||\cA|}{\delta\cdot\epsilon}\cdot\log\frac{H}{\epsilon}\bigg]
	\]
	and obtains a policy $\pi$ such that, with probability at least $1-\delta$,
	\[
	\forall h\in[H]:\bv_h^* - \epsilon\one\le \bv_h^{\pi} \le \bv^*_h,
	\]
	where $\bv_h^{*}$ is the optimal value of $\cM$ at stage $h$.
	Moreover, the number of samples used by the algorithm is
	\[
	O\bigg[{\frac{1}{\epsilon^2}}\cdot{H^3|\cS||\cA|}\cdot \log\frac{H|\cS||\cA|}{\delta\cdot\epsilon}\cdot\log\frac{H}{\epsilon}\bigg].
	\]
\end{theorem}

\subsection{Sample Lower Bound On $H$-MDP}
In this section we show that the sample complexity obtained by the algorithm in the last section is essentially tight.
Our proof idea is simple, we will reduce the $H$-MDP problem to a discounted MDP problem.
If there is an algorithm that solves an $H$-MDP to obtain an $\epsilon$-optimal value, it also gives an value function to the discounted MDP. 
Therefore, the lower bound of solving $H$-MDP inherits from that of the discounted MDP. 
The formal guarantee is presented in the following theorem.

\begin{theorem}
	Let $\cS$ and $\cA$ be finite sets of states and actions. Let $H>0$ be a positive integer and $\epsilon\in (0,1/2)$ be an error parameter.
	Let $\cK$ be an algorithm that, on input an $H$-MDP $\cM\defeq(\cS, \cA, P, \br)$ with a sampling oracle, outputs a value function $\bv_1$ for the first stage, such that $\|\bv_1 - \bv^*_1\|_{\infty}\le \epsilon$ with probability at least $0.9$.
	Then $\cK$ calls the sampling oracle at least $\Omega(H^{-3}\epsilon^{-2}|\cS||\cA|/\log\epsilon^{-1})$ times on  some input $P$ and $\br\in [0,1]^{\cS}$.
\end{theorem}
\begin{proof}
	Let $s_0\in \cS$ be a state. 
	Denote $\cS' = \cS\backslash\{s_0\}$ be a subset of $\cS$. 
	Let $\gamma\in (0, 1)$ be such that $(1-\gamma)^{-1}\log\epsilon^{-1} \le H$.
	Suppose we have an DMDP $\cM' = (\cS', \cA, P', \gamma, \br')$ with a sampling oracle. 
	Let $\bv^{*'}$ be the optimal value function of $\cM'$. 
	Note that $\bv^{*'}\in \RR^{\cS'}$. 
	We will show, in the next paragraph, an $H$-MDP $\cM = (\cS, \cA, P, H, \br)$ with first stage value $\bv_1^*$, such that $\|\bv_1^*|_{\cS'}-\bv^{*'}\|\le \epsilon$.
	Therefore, an $\epsilon$-approximation of $\bv_1^*$ gives a $2\epsilon$-approximation to $\bv^*$.
	We show that $\cK$ can be used to obtain an $\epsilon$-approximate value $\bv_1$ for $\bv_1^*$ of $\cM$ and thus $\cK$ inherits the lower bound for obtaining $(2\epsilon)$-approximated value for $\gamma$-DMDPs.
	
	For $\cM$, in each state $s\in \cS'$, for any action there is a $(1-\gamma)$ probability transiting to $s_0$ and $\gamma$ probability to do the original transitions in $\cM'$; for $s_0$, no matter what action taken, it transits to itself with probability $1$. 
	Formally, for each state $s,s'\in \cS', a\in \cA$, $P(s'|s, a) = \gamma\cdot P'(\cdot|s,a)$ and $P(s_0|s, a) = (1-\gamma)$; $P(s'|s_0, a) = 0$ and $P(s_0|s_0, a) = 1$.
	For $\br$, we set $\br(s_0, \cdot) ={\bf0}$ and $\br(s, \cdot) = \br'(s, \cdot)$ for $s\in \cS'$.
	It remains to show that $\|\bv_1^*|_{\cS'} - \bv^*\|_\infty\le \epsilon$.
	First we note that $\bv(\br) = \bv_{H}^* \le \bv^*$.
	Then, by monotonicity of the $\cT$ operator, we have, for all $h\in [H-1]$ and $s\in\cS'$,
	\begin{align*}
		\bv_{h}^*|_{\cS'}(s) = \max_a[\br'(s, a) + \gamma\bP_{s,a}^{'\top} \bv_{h+1}^*] \le \bv^{*'}.
	\end{align*}
	In particular, $\bv_{1}^*|_{\cS'}\le \bv^{*'}$.
	Since the optimal policy $\pi^{*'}$ of $\cM'$ can be used as a policy for the $H$-MDP as a non-optimal one, we have
	\[
		 \bv^* - \epsilon\cdot \one\le \bigg[1 + \gamma\bP_{\pi^{*'}} + \gamma^2\bP_{\pi^{*'}}^2+\cdot +  \gamma^{H}\cdot \bP_{\pi^{*'}}^{H}\bigg]\br^{\pi^{*'}} \le  \bv_1^*|_{\cS'}.
	\] 
	This completes the proof.
	%1) add a new state and modify transitions so that each state goes to the new state with probability (1 - gamma) and transitions as the would have been for with probability gamma. also make new state transition only to itself and make reward to go to new state and state there 0. Note that now the total value for infinite horizon from a state is the total discounted value
	%2) suppose we had an algorithm that solved this to horizon H = 1/(1 - gamma) log(1/eps) giving an epsilon accurate policy
	%3) run Lin and Carrie's finite horizon algorithm (the one that is sample effecient but not time effecient) on this policy. This should let us get the value of the original states in the policy to eps accuracy with $O(|S| H^3 / eps^2) = O(|S| (1 - \gamma)^{-3} eps^{-2})$ samples
\end{proof}
The above lower bound with our algorithm also implies a sample lower bound for an $\epsilon$-policy. 
\begin{corollary}
Let $\cS$ and $\cA$ be finite sets of states and actions. Let $H>0$ be a positive integer and $\epsilon\in (0,1/2)$ be an error parameter.
Let $\cK$ be an algorithm that, on input an $H$-MDP $\cM:=(\cS, \cA, P, \br)$ with a sampling oracle, outputs a policy $\pi: \cS\times[H]\rightarrow \cA$, such that $\forall h: \|\bv^{\pi}_h - \bv^*_h\|_{\infty}\le \epsilon$ with probability at least $0.9$.
Then $\cK$ calls the sampling oracle at least $\Omega(H^{-3}\epsilon^{-2}|\cS||\cA|/\log\epsilon^{-1})$ times on the worst case input $P$ and $\br\in [0,1]^{\cS}$.
\end{corollary}

\end{document}